\documentclass{amsart}
\usepackage[utf8]{inputenc}
\usepackage{amsfonts, amsthm, amsmath, mathtools, amssymb, amsxtra, graphicx, comment, xcolor, enumitem, mathabx,xfrac, mathrsfs, indentfirst, esint, bm, tensor, setspace, geometry, comment}
\usepackage[frak=esstix]{mathalfa}

\newtheorem{lemma}{Lemma}[section]
\newtheorem{thm}{Theorem}[section]

\newtheorem{rmk}{Remark}[section]

\newcommand{\N}{\mathbb{N}}
\newcommand{\Zbb}{\mathbb{Z}}

\newcommand{\R}{\mathbb{R}}
\newcommand{\C}{\mathbb{C}}
\newcommand{\F}{\mathscr{F}}
\newcommand{\T}{\mathbb{T}}
\newcommand{\grad}{\nabla}
\newcommand{\p}{\partial}
\newcommand{\disc}{\mathbb{D}}
\newcommand{\D}{\Delta}
\newcommand{\rot}{\mathscr{R}}
\newcommand{\W}{\Omega}
\newcommand{\CD}{\mathcal{D}}

\newcommand{\al}{\alpha}

\newcommand{\be}{\beta}
\newcommand{\y}{\gamma}

\newcommand{\eps}{\varepsilon}
\newcommand{\kap}{\kappa}
\newcommand{\lam}{\lambda}
\newcommand{\Lam}{\Lambda}

\newcommand{\sig}{\sigma}
\newcommand{\tta}{\theta}

\newcommand{\z}{\zeta}

\newcommand{\Proj}{\mathbf{P}}

\newcommand{\id}{\mathrm{id}}

\newcommand{\Rea}{\mathfrak{Re}}
\newcommand{\Ima}{\mathfrak{Im}}

\newcommand{\norm}[1]{\left\Vert#1\right\Vert}

\newcommand{\pr}{\prime}
\newcommand{\Z}{\mathcal{Z}}

\newcommand{\Zbar}{\overline{Z}}
\newcommand{\zebar}{\overline{\z}}

\newcommand{\Hom}{\mathrm{Hom}}
\newcommand{\Cau}{\mathfrak{C}}
\newcommand{\Hil}{\mathfrak{H}}
\newcommand{\uhat}{\hat{u}}
\newcommand{\xitil}{\tilde{\xi}}
\newcommand{\lamtil}{\tilde{\lambda}}

\newcommand{\Range}{\mathfrak{R}}

\newcommand{\abs}[1]{\left\lvert#1\right\rvert}

\newcommand{\ek}[1]{\mathrm{e}_{#1}}
\newcommand{\Ztil}{\widetilde{Z}}
\newcommand{\ctil}{\widetilde{c}}

\newcommand{\set}[1]{\left\{#1\right\}}

\newcommand{\ZhatAlt}

\newcommand{\Curveloc}{\mathcal{C}_{\mathrm{loc}}}
\newcommand{\Curve}{\mathcal{C}}

\newcommand{\open}{\mathscr{G}}

\newcommand{\normX}[3]{\norm{#1}_{X^{#2,#3}}}
\newcommand{\pnorm}[2]{\norm{#1}_{L^{#2}(\T)}}

\newcommand{\nbhd}{\mathscr{N}}
\newcommand{\CA}{\mathscr{C}}

\DeclareMathOperator{\absD}{\abs{D_\al}}

\DeclareMathOperator{\bigo}{\mathcal{O}}
\DeclareMathOperator{\ind}{\mathrm{ind}}

\numberwithin{equation}{section}
\title{Global Bifurcation of Steady Surface Capillary Waves on a 2D Droplet}
\author{Gary Moon and Yilun Wu}
\date{February 2024}

\begin{document}

\maketitle

\begin{abstract}
    %We consider the dynamics and evolution of a capillary drop; that is, we study the well-known water waves problem posed on the surface of a droplet as opposed to on a fluid layer. We construct $2d$ traveling waves possessing certain rotational and reflection symmetries. The primary tool we utilize is a real analytic variant of global bifurcation theory which allows us to construct a global locally real analytic curve of such solutions.
    We construct global curves of rotational traveling wave solutions to the 2D water wave equations on a compact domain. The real analytic interface is subject to surface tension, while gravitational effects are ignored. In contrast to the rotational surface waves, the fluid follows the incompressible, irrotational Euler equations. This model can provide a description for tiny water droplets in breaking waves and white caps. The primary tool we use is global bifurcation theory, via a conformal formulation of the problem. The obtained fluid domains have $m$-fold discrete rotational symmetry, as well as a reflection symmetry.
\end{abstract}

\section{Introduction}

%{\color{red} Reorganize presentation of literature review: Emphasize progression from small-amplitude linearized waves to global connected sets of solutions with singularity in the limit. Current version emphasizes Stokes waves with $g>0,\sig=0$. Also, discuss Crapper waves $g=0$,$\sig>0$, which is more analogous to our case. Discuss a bit more about other results with surface tension including the special difficulties it brings.}

In this paper we construct global curves of solutions to the two-dimensional free boundary incompressible, irrotational Euler equations. As is well-known, this problem can be reduced to the following free surface Bernoulli equations (\cite{Lan1}). In this formulation, one looks for a time dependent fluid domain $\mathcal D(t)\subset \R^2$ bounded by a sufficiently smooth non-self-intersecting curve 
\begin{equation}\label{Bern 1}
    \mathcal S(t)=\{(x(s,t),y(s,t)):s\in \R\},
\end{equation}
and a sufficiently smooth function $\varphi(x,y,t)$ on $\overline{\mathcal D(t)}$ such that 
\begin{align}
    \Delta\varphi = 0 \quad &\text{ on }\mathcal D(t),\label{Bern 2}\\
    (x_t,y_t)\cdot n = \nabla \varphi \cdot n \quad &\text{ on }\mathcal S(t), \label{Bern 3}\\
    \varphi_t+\tfrac12|\nabla \varphi|^2 +gy -\sigma \kappa = Q\quad &\text{ on }\mathcal S(t).\label{Bern 4}
\end{align}
Here $n$ is the outward unit normal vector on $\mathcal S(t)=\partial \mathcal D(t)$, $\kappa$ the curvature on $\mathcal S(t)$, $g>0$ is the gravity constant, $\sigma>0$ is the capillarity constant, and $Q$ is the Bernoulli constant, which may be absorbed into $\varphi$ by adding a $t$-dependent function to it. There is a large literature devoted to the Cauchy problem of \eqref{Bern 1}-\eqref{Bern 4}. The current paper, however, is a contribution to the study of steady solutions to the same system, which by itself is an active area of research with a long history. To gain a better focus on the topic, we restrict ourselves mostly to discussing steady solutions in the following introduction.

The basic steady solutions studied in the literature are translational traveling waves. These are solutions to \eqref{Bern 1}-\eqref{Bern 4} of the form
\begin{align}
    (x(s,t),y(s,t))&=(\tilde x(s)-ct,\tilde y(s)),\\
    \varphi(x,y,t)&=\tilde\varphi(x-ct,y).
\end{align}
Equivalently, one looks for a domain $\tilde {\mathcal D}$ bounded by a non-self-intersecting curve
\begin{equation}
    \tilde {\mathcal S}=\{(\tilde x(s),\tilde y(s): s\in \R\}
\end{equation}
and a function $\tilde\varphi$ on $\overline {\tilde {\mathcal D}}$ such that
\begin{align}
    \Delta\tilde \varphi = 0 \quad &\text{ on }\tilde {\mathcal D},\\
     \nabla \tilde \varphi \cdot n=(-c,0)\cdot n \quad &\text{ on }\tilde{\mathcal S} , \\
    -c\tilde \varphi_x+\tfrac12|\nabla \tilde \varphi|^2 +gy -\sigma \kappa = Q\quad &\text{ on }\tilde {\mathcal S} .
\end{align}
%Our solutions represent  traveling water waves propagating, under the influence of capillarity, on the surface of a droplet of an incompressible, inviscid, homogeneous fluid. This paper is then a contribution to the study of steady solutions to the water waves system, an active area of research in fluid mechanics with a long history. {\color{red} On the other hand, the study of general Cauchy problems for water waves also produced a rich literature, of which this introduction does not cover.}

The mathematical study of water waves goes back over two centuries to work of Laplace and Lagrange in the latter part of the eighteenth century. Some early work on traveling water waves involved studying the linearized (about still water) equations and was done by mathematicians such as Cauchy, Poisson and Airy. A quite remarkable early contribution was made by Gerstner, who discovered a family of exact steady solutions to the full nonlinear system. Another impressive advancement was that of Stokes when he utilized a perturbation expansion (in the wave amplitude) to study traveling water waves. Stokes also considered waves of large amplitude, having famously conjectured the existence of the wave of extreme form (or wave of greatest height), which has sharp crests of included angle $\frac{2\pi}{3}$ \cite{Sto1}.

The first rigorous construction of small-amplitude traveling waves was obtained by Nekrasov \cite{Nek1}. His approach was to use a sort of Stokes expansion and show that the series had a positive radius of convergence. Also notable regarding the work of Nekrasov is his formulation of the problem as an integral equation on the boundary of the fluid domain. This formulation has been quite influential and many studies have made use of his formalism. Another important contribution was made by Levi-Civita who used a similar approach to Nekrasov (expanding in powers of the amplitude and showing convergence), but formulated the problem differently \cite{Lev1}.

Up until this point, perturbative assumptions (i.e., small amplitude) were absolutely critical to the analysis. The first large-amplitude traveling water waves were not constructed until the middle of the twentieth century in the work of Krasovski\u{i} \cite{Kra1}. By utilizing the positive operator theory of Krasnoselskii, Krasovski\u{i} was able to construct a set of large-amplitude traveling waves.

A major milestone in the study of traveling water waves was the introduction of a general theory of bifurcation from a simple eigenvalue \cite{CrRa1} and global continuation (e.g., Dancer's theory of global bifurcation for analytic operators \cite{Dan1}), along with the realization that the problem of traveling water waves fit into this framework. Using such tools often allows one to not simply construct individual traveling wave solutions or even sets of solutions, but global connected sets (or even curves) of solutions.

Indeed, Keady-Norbury used this theory to construct the first global (smooth) curve of large-amplitude traveling water waves \cite{KeNo1}. It was through studying the limiting behavior of solutions along this curve that the famed Stokes conjecture on the wave of extreme form \cite{AFT1,Tol2} was resolved. Further constructions of global connected sets of (translational) traveling gravity water waves ($g>0$ and $\sig=0$) can be found, for example, in \cite{CoSt1, CSV1, AmTo1}. A good survey of the literature on traveling water waves can be found in \cite{HazEtAl1}.

Traveling waves in the presence of surface tension are quite a bit less well-understood than in the case of pure gravity waves. An important early contribution was made by Crapper, who derived a global family of exact solutions to the full nonlinear Euler equations describing traveling capillary waves ($g=0$ and $\sig>0$) over an infinitely-deep fluid layer \cite{Cra1}. These solutions, so-called the Crapper waves, are well known to form splash singularities in the troughs of the waves, as the continuation parameter approaches a limiting value. A similar construction was carried out in the case of finite depth by Kinnersley \cite{Kin1}.

Much like with traveling gravity waves or capillary waves, bifurcation analysis has been a powerful tool for constructing more general traveling waves subject to a combination of gravity, surface tension ($g>0$ and $\sigma>0)$, and other effects such as voriticity and surface wind. For example, for construction of small-amplitude waves, see \cite{CrNi1, Wah1} and the references therein. For large-amplitude waves and criteria of singularity formation in global solution sets, see \cite{Walsh1, CSV1} and included references.

The primary difference between our work and other works on the subject of traveling water waves is that we consider {\em rotational} waves on the surface of a capillary drop as opposed to translational waves on the surface of a fluid layer. More precisely, we look for $g=0$ solutions to \eqref{Bern 1}-\eqref{Bern 4} of the form 
\begin{align}
(x(\alpha,t),y(\alpha,t))&= R_{ct}(\tilde x(\alpha),\tilde y(\alpha)),\\
\varphi(x,y,t)&=\tilde\varphi(R_{-ct}(x,y)),
\end{align}
where $\alpha\in\R_{\text{mod } 2\pi}$, $R_\theta(x,y)=(x\cos\theta-y\sin\theta, x\sin\theta+y\cos\theta)$. Incidentally, the boundary curve is a simple closed curve bounding a compact domain (the droplet), rather than the typical case of an open curve with an unbounded fluid domain in the above mentioned literature. We also emphasize here that, although the surface waves make the boundary of the fluid domain rotate rigidly, the fluid flow itself is actually irrotational. In fact, the inside of the fluid must follow some complicated flow far from a rigid body rotation, to create the appearance of a rotational wave on the surface.

We can describe the rotational traveling wave problem more directly as follows. We look for a compact domain $\mathcal D$ bounded by a simple closed curve 
\begin{equation}\label{Bern rot 1}
    \mathcal S=\{(x(\alpha),y(\alpha)):\alpha\in \R_{\text{mod } 2\pi}\}
\end{equation}
and a function $\varphi$ on $\overline {\mathcal D}$ such that
\begin{align}
    \Delta \varphi &=0 \quad \text{ on }\mathcal D, \label{Bern rot 2}\\
    (c~(-y,x)- \nabla \varphi )\cdot n&=0 \quad \text{ on }\mathcal S, \label{Bern rot 3}\\
    -c~(-y,x)\cdot \nabla \varphi +\tfrac12|\nabla\varphi|^2-\sigma\kappa &= Q \quad \text{ on }\mathcal S. \label{Bern rot 4}
\end{align}

One of the earliest contributions to the study of the evolution of waves on the surface of a capillary drop was made by Lord Rayleigh. Indeed, in \cite{Ray1}, he computed small oscillations of a capillary drop near equilibrium. Beyer-G\"{u}nther used a coordinate-free geometric formulation of the system to demonstrate that the problem is well-posed in Sobolev spaces, the earliest result on the well-posedness of the water waves system in the presence of surface tension \cite{BeGu1}. In \cite{Dya1}, Dyachenko undertakes a numerical study of the motion of a capillary drop as a means to better understand the process of whitecapping (see also \cite{DyNe1}) and how a droplet might fracture into a spray. This work was continued in \cite{Dya2} and \cite{ChDy1}, where traveling waves on a capillary drop were numerically computed with an eye to studying the limiting behavior.

Though traveling waves on the surface of a capillary drop have been computed via numerical and perturbation analysis, their existence and behavior have not been rigorously demonstrated. It is our goal in this paper to initiate the rigorous study of this phenomenon by proving their existence. 
Indeed, our main theorem can be stated as follows. 
\begin{thm}
    \label{thm:InformalMainThm}
Fix $\sigma>0$, a H\"{o}lder coefficient $\be\in(0,1)$, a symmetry multiplicity $m\in \N$, $m\geq2$, and an index $k\in \N$. Then there exists a global curve of solutions $(\tilde x(s,\cdot),\tilde y(s,\cdot), \tilde c(s,\cdot))$, such that for each $s\in \R$, 
$(x(\alpha),y(\alpha))=(\tilde x(s,\alpha)$, $\tilde y(s,\alpha))$, $\alpha\in \R_{\text{mod } 2\pi}$ is a simple closed $C^{2,\beta}$ curve, and $c=\tilde c(s)$ a real number. They give a solution to \eqref{Bern rot 1}-\eqref{Bern rot 4}. i.e. on the domain $\mathcal D$ bounded by $\mathcal S=\{(x(\alpha),y(\alpha)):\alpha\in \R_{\text{mod } 2\pi}\}$, there exists a $\varphi\in C^{2,\beta}(\overline {\mathcal D})$ and $Q\in \R$ so that \eqref{Bern rot 1}-\eqref{Bern rot 4} hold. The domain $D$ satisfies an $m$-fold rotational symmetry and a reflection symmetry:
\begin{equation}
    \left(x\left(\alpha+\tfrac{2\pi}m\right),y\left(\alpha+\tfrac{2\pi}m\right)\right)=R_{\frac{2\pi}{m}}(x(\alpha),y(\alpha)),~(x(-\alpha), y(-\alpha))=(x(\alpha),-y(\alpha)).
\end{equation}
Furthermore, the curve $(\tilde x, \tilde y, \tilde c)$ satisfies the following:
\begin{equation}
\tilde c(0)=\pm\sqrt{ mk - \frac{1}{mk} },
\end{equation}
and either one of the following alternatives occurs:
\begin{enumerate}[label=(\alph*)]
\item $C^1$ norm blow-up:
\begin{equation}
    \lim_{|s|\to\infty}\big\|(\tilde x(s,\cdot),\tilde y(s,\cdot))\big\|_{[C^1(\R_{\text{mod } 2\pi})]^2}=\infty;
\end{equation}
\item chord-arc bound blow-up:
\begin{equation}
    \lim_{|s|\to\infty} \sup_{\alpha_1\ne\alpha_2}\frac{|\alpha_1-\alpha_2|_{\text{mod } 2\pi}}{|(\tilde x(s,\alpha_1),\tilde y(s,\alpha_1))-(\tilde x(s,\alpha_2),\tilde y(s,\alpha_2))|}=\infty;
\end{equation}
\item closed loop, i.e., there exists some $s_0>0$, such that for all $s\in \R$,
\begin{equation}
    (\tilde x(s+s_0,\cdot),\tilde y(s+s_0,\cdot), \tilde c(s+s_0))=(\tilde x(s,\cdot),\tilde y(s,\cdot), \tilde c(s)).
\end{equation}
\end{enumerate}
\end{thm}

\begin{rmk}
    \label{rmk:InformalMainThm}
    \begin{enumerate}
        \item For $\alpha\in \R_{\text{mod } 2\pi}$, $|\alpha|_{\text{mod } 2\pi}=\displaystyle \min_{n\in \mathbb Z}|\alpha-2n\pi|$ is the distance to $0$ in $\R_{\text{mod } 2\pi}$.
        \item Theorem \ref{thm:InformalMainThm} is a result of global bifurcation theory. For a more precise version of Theorem \ref{thm:InformalMainThm} including the bifurcation statement, see Theorem \ref{thm:MainLocal} and Theorem \ref{thm:MainGlobal}. We use a conformal reformulation for the more technical version of the theorems. 
        \item For clarity, and to reduce notational clutter, we construct the curve of solutions with the fluid domain having $C^{2,\be}$ regularity. One can easily modify the proof to work with $C^{k,\be}$ regularity for $k>2$ instead. In fact, the constructed fluid domains all have real analytic regularity. See Section \ref{sec: ana} for more details. 
        \item A typical way $C^1$ norm blow-up can occur is for the boundary of fluid domains to form a non-smooth jump or ``fracture".
        \item A typical way chord-arc bound blow-up can occur is for the fluid domains to gradually corrugate and form air bubbles.
        \item In other models without surface tension (mainly with translational traveling waves), the loop alternative can often be eliminated via ``nodal analysis" and maximum principle type arguments. %However, such methods break down when surface tension is introduced. The higher derivatives of the curvature term cause the analyis to fail. As far as the authors are aware, all known global bifurcation/continuation results allowing surface tension leave the loop alternative open. {\color{red} provide some references?} 
        %However, when it comes to the construction of global curves of solutions, the presence of surface tension poses some serious new challenges. 
%In the case of global curves of traveling gravity water waves, it is generally possible to employ a maximum-principle-type argument to demonstrate that the curve of solutions is not a closed curve. 
However, in the presence of surface tension, the higher number of derivatives from the curvature term destroys one's ability to exploit such arguments. As a result, the loop possibility is often simply retained (e.g., see \cite{Walsh1}). This has made it a difficult task to guarantee norm blow-up for limiting large-amplitude waves in the presence of surface tension.

The retention of the loop is not the only delicate issue that arises in the study of traveling water waves due to the presence of surface tension. This is particularly the case when one considers gravity-capillary waves ($g>0$ and $\sig>0$). Fixing $g>0$, there are certain values of $\sig>0$ which result in the linearized operator having a higher-dimensional null space (i.e., resonance). For such values of $(g,\sig)$, various kinds of bifurcations (and secondary bifurcations) can occur. This is the source of so-called Wilton ripples. For more on such phenomena, see \cite{ToJo1}.

        \item This theorem does not provide details describing a limiting shape showing how one of the above alternatives actually occurs, although some numerical computations hint at possible corrugation of fluid boundaries and formation of air bubbles. We leave this to future studies. 
        %\item Our approach will work just as well, with only minor modifications, in the Sobolev space setting. That is, our proof implies the existence of such a curve $\Curve$ in $H^r(\T)\times\R$, $r\geq2$.
    \end{enumerate}
\end{rmk}

Our approach employs the conformal formalism of \cite{Dya1}. Namely, we shall study (time-dependent) Riemann maps $z(w,t)$ from the unit disk $\disc$ into the fluid domain $\CD(t)$. Given our interest in traveling waves, we seek these conformal maps under the form
\[
z(w,t) = Z(w)e^{ict}
\]
for a fixed Riemann map $Z$ and wave speed $c$. Moreover, we assume our Riemann map $Z$ to verify certain symmetries. More particularly, we suppose that, for $m\geq2$ an integer,
\[
Z(e^{i\frac{2\pi}{m}}w) = e^{i\frac{2\pi}{m}}Z(w) \text{ and } \overline{Z(w)} = Z(\overline{w}).
\]
Observe that $m$ represents a symmetry multiplicity. For example, a traveling wave corresponding to the Riemann map $Z$ with $m=3$ will have 3 symmetric lobes.

We show, in Section 2, that such a (fixed) Riemann map $Z$ must be a solution to the equation
\begin{equation}
    \label{eqn:TravelingWaveEqn-Intro}
2\Cau\left(\frac{Z_\al}{\abs{Z_\al}}\right)_\al - i2Z_\al + c^2\Cau( Z\Hil( \abs{Z}^2 )_\al ) = 0.
\end{equation}
Note that the trivial solution given by $Z_0(w)=w$ is a solution to \eqref{eqn:TravelingWaveEqn-Intro}.

The remainder of the paper is dedicated to the proof of Theorem \ref{thm:InformalMainThm} and is organized as follows. In Section 2, we begin by writing down the relevant equations of motion and then using holomorphic coordinates/the conformal method to obtain a convenient formulation of the problem. Next, in Section 3, we carry out the necessary setup to apply the Dancer/Buffoni-Toland framework of analytic global bifurcation theory to the problem. In Section 4, we use this analytic bifurcation theory to construct local curves of solutions bifurcating from the trivial (circular) solution. In Section 5, we complete the proof of Theorem \ref{thm:InformalMainThm} by continuing the local curves constructed in Section 4 to global curves of solutions. Moreover, in Section \ref{sec: ana}, we show that the obtained fluid domains have real analytic boundaries.

\section{Holomorphic Coordinates}
In this section, we derive the reformulation of the problem in holomorphic coordinates, and prove a few lemmas along the way that will become useful later on. We consider a 2D unit density incompressible irrotational fluid on a simply connected moving domain $\mathcal D(t)$ with a sufficiently smooth free boundary subject to surface tension. We ignore gravity and other effects in the fluid and model it by the standard incompressible Euler equations. By a Galilean transformation, we may assume without loss of generality that the center of mass of $\mathcal D(t)$ is fixed at the origin. Let $z=x+iy$ be the complex coordinates in physical space. Let $\mathbb D$ be the unit disc centered at the origin with $\mathbb T=\partial \mathbb D$. Denote by $w=u+iv$ the complex coordinates on $\mathbb D$. Let $z=z(w,t)$ be a Riemann mapping from $\mathbb D$ to $\mathcal D(t)$ satisfying $z(0,t)=0$. It can be continuously extended to a diffeormorphism from $\overline{\mathbb D}$ to $\overline{\mathcal D(t)}$. Such a Riemann mapping is not unique, but the non-uniqueness is only up to a rotation on $\mathbb D$. It will be uniquely determined if one specifies $z(1,t)\in \partial\mathcal D(t)$. Let $\varphi(x,y,t)$ be a harmonic velocity potential on $\mathcal D(t)$. It is well known that the Euler equations on $\mathcal D(t)$ can be reduced to the kinematic boundary condition
\begin{equation}\label{original KBC}
    (x_t,y_t)\cdot n = \nabla \varphi \cdot n \quad \text{ on }\partial \mathcal D(t),
\end{equation}
and the dynamical boundary condition
\begin{equation}\label{original DBC}
    \varphi_t+\frac12|\nabla \varphi|^2+\sigma \kappa =0\quad \text{ on }\partial \mathcal D(t).
\end{equation}
Here $n$ is the unit normal vector and $\kappa$ is the curvature on $\partial\mathcal D(t)$. $\sigma$ is the capillarity constant. We can rewrite these equations in complex coordinates. Denote by 
\begin{equation}\label{def: phi}
    \phi(w,t)=\varphi(x(u,v,t),y(u,v,t),t)
\end{equation}
the pulled-back velocity potential on $\mathbb D$, and let $\theta(w,t)$ be a conjugate harmonic function of $\phi$ on $\mathbb D$. $\theta$ is the pulled-back stream function. Denote by $\Phi=\phi+i\theta$ the pulled-back holomorphic potential on $\mathbb D$. Note that by construction and by symmetry of \eqref{original KBC} and \eqref{original DBC}, we can add $c_1+ic_2(t)$ to $\Phi$ and still get a solution to the problem. Here $c_1$ is an arbitrary real constant, and $c_2(t)$ is an arbitrary real-valued function. Let $\alpha$ be the polar angular variable on $\mathbb D$. Note that the $\alpha$-derivative of any holomorphic function $f(w)$ on $D$ is given by
\begin{equation}\label{d_al}
    f_\alpha(w)=-yf_x+xf_y=iwf'(w).
\end{equation}
Noting $(a,b)\cdot (c,d)=\Rea [(a+ib)\overline{(c+id)}]$, \eqref{original KBC} can be rewritten as  
\begin{equation}\label{KBC 1}
    \Rea(z_t \overline{iz_\alpha})=\Rea\left(\overline{\left(\frac{\Phi'}{z'}\right)}\overline{iz_\alpha}\right)\quad \text{ on }\mathbb T.
\end{equation}
Here we denoted the $w$-derivatives by $'$ and used 
\begin{equation}\label{Phi'}
    \Phi'=\Phi_z z' = (\varphi_x-i\varphi_y)z'.
\end{equation}
By \eqref{d_al}, we obtain from \eqref{KBC 1} 
\begin{equation}\label{KBC on T}
    \Ima(z_t\overline{z_\alpha})=\Ima(\overline{\Phi_\alpha})\quad \text{ on }\mathbb T.
\end{equation}
From \eqref{def: phi} and \eqref{Phi'} we get
\begin{align}
    \phi_t&=\varphi_x x_t+\varphi_y y_t +\varphi_t \notag\\
    &= \varphi_t +\Rea [(x_t+iy_t)(\varphi_x-i\varphi_y)]\notag\\
    &=\varphi_t +\Rea\left(z_t\frac{\Phi'}{z'}\right).\label{varphi_t}
\end{align}
On the other hand, we have
\begin{align}
    \kappa = \frac{x_\alpha y_{\alpha\alpha}-y_\alpha x_{\alpha\alpha}}{|z_\alpha|^3} &=\frac{\Ima (\overline{z_\alpha }z_{\alpha\alpha})}{|z_\alpha|^3}\notag\\
    &=\frac{z_{\alpha\alpha}\overline{z_\alpha}-\overline{z_{\alpha\alpha}}z_\alpha}{2i|z_\alpha|^3}\notag\\
    &=\frac{1}{iz_\alpha}\left(\frac{z_\alpha}{|z_\alpha|}\right)_\alpha.\label{kappa 1}
\end{align}
By \eqref{Phi'}, \eqref{varphi_t} and \eqref{kappa 1}, \eqref{original DBC} can be written as 
\begin{equation}\label{DBC 1}
    \Rea \Phi_t - \Rea\left(z_t\frac{\Phi'}{z'}\right)+\frac12\left|\frac{\Phi'}{z'}\right|^2+\frac{\sigma}{iz_\alpha}\left(\frac{z_\alpha}{|z_\alpha|}\right)_\alpha=0\quad \text{ on }\mathbb T.
\end{equation}
We can further simplify the middle two terms in \eqref{DBC 1} as follows. Using \eqref{d_al} and \eqref{KBC on T}, we have
\begin{align}
    &~\frac12\left|\frac{\Phi'}{z'}\right|^2- \Rea\left(z_t\frac{\Phi'}{z'}\right)\notag\\
    =&~ \frac12\left|\frac{\Phi_\alpha}{z_\alpha}\right|^2- \Rea\left(z_t\frac{\Phi_\alpha}{z_\alpha}\right)\notag\\
    =&~\frac1{2|z_\alpha|^2}\Rea(|\Phi_\alpha|^2-2\Phi_\alpha z_t\overline{z_\alpha})\notag\\
    =&~\frac1{2|z_\alpha|^2}[\Phi_\alpha \overline{\Phi_\alpha}-\Phi_\alpha z_t\overline{z_\alpha}-\overline{\Phi_\alpha}\overline{z_t}z_\alpha]\notag\\
    =&~\frac1{2|z_\alpha|^2}[(\overline{\Phi_\alpha})^2-\overline{\Phi_\alpha}2i\Ima(\overline{\Phi_\alpha}) -\Phi_\alpha z_t\overline{z_\alpha}-\overline{\Phi_\alpha}\overline{z_t}z_\alpha]\notag\\
    =&~\frac1{2|z_\alpha|^2}[(\overline{\Phi_\alpha})^2-\overline{\Phi_\alpha}(\overline{z_t}z_\alpha+2i\Ima(z_t\overline{z_\alpha})) -\Phi_\alpha z_t\overline{z_\alpha}]\notag\\
    =&~\frac1{2|z_\alpha|^2}[(\overline{\Phi_\alpha})^2-(\overline{\Phi_\alpha}+\Phi_\alpha) z_t\overline{z_\alpha}]\notag\\
    =&~\frac1{2z_\alpha}\left[\frac{(\overline{\Phi_\alpha})^2}{\overline{z_\alpha}}-2z_t\Rea \Phi_\alpha\right].
\end{align}
Thus \eqref{DBC 1} can be written as 
\begin{equation}\label{DBC 2}
    \frac{1}{z_\alpha}\left[z_\alpha \Rea \Phi_t-z_t\Rea \Phi_\alpha-i\sigma\left(\frac{z_\alpha}{|z_\alpha|}\right)_\alpha+\frac{(\overline{\Phi_\alpha})^2}{2\overline{z_\alpha}}\right]=0\quad \text{ on }\mathbb T.
\end{equation}
To further simplify the equation, we introduce the Cauchy integral operator. For $f\in C^{0,\be}(\mathbb T)$ with $0<\be<1$, we define its Cauchy integral by
\begin{equation}
    \Cau f(w)=\frac1{2\pi i}\int_\T \frac{f(\tau)}{\tau-w}~d\tau,
\end{equation}
and denote its average by 
\begin{equation}
    \underset{\T}{\text{Avg}}f = \frac1{2\pi}\int_0^{2\pi}f(e^{i\alpha})~d\alpha = \frac{1}{2\pi i}\int_\T \frac{f(\tau)}{\tau}~d\tau .
\end{equation}
For simplicity, we sometimes also write $\Cau f$ instead of $\Cau(f\big|_\T)$ for $f$ defined on $\overline{\mathbb D}$. Note that $\Cau f$ is a holomorphic function on $\mathbb D$ with continuous boundary values in $C^{0,\be}(\T)$. In fact, for any $\tau_0=e^{i\theta_0}\in\T$, we have the following Plemelj formula:
\begin{equation}\label{Plemelj original}
    \lim_{\substack{w\to \tau_0 \\ w\in \mathbb D}}\Cau f(w)=\frac12f(\tau_0)+\frac1{2\pi i}\int_\T \frac{f(\tau)}{\tau-\tau_0}~d\tau=\frac12\underset{\T}{\text{Avg}}f +\frac12f(\tau_0)+\frac i2\mathfrak H f(\tau_0).
\end{equation} 
Here 
\begin{equation}
    \mathfrak H f(e^{i\theta_0})=\frac1{2\pi}\int_0^{2\pi}f(e^{i\theta}) \cot\frac{\theta_0-\theta}{2}~d\theta
\end{equation}
is the Hilbert transform on the circle. The singular integrals in the above formulas are all interpreted as Cauchy principle values. In the following, we often abuse notation and regard $\Cau f$ as defined on $\overline{\mathbb D}$. Note that $\Hil f$, however, is a function on $\T$. Thus we may write \eqref{Plemelj original} as
\begin{equation}\label{Plemelj}
    (\Cau f)\big|_\T = \frac12\underset{\T}{\text{Avg}}f +\frac12f+\frac i2\mathfrak H f.
\end{equation}
Similar to the convention for Cauchy integral, we sometimes write $\Hil f$ instead of $\Hil (f\big|_\T)$ for a function $f$ defined on $\overline{\mathbb D}$. The Cauchy integral formula and the Plemelj formula imply the following Titchmarsh type theorem.
\begin{lemma}\label{lem: Titchmarsh}
    Let $f$ be a holomorphic function on $\mathbb D$ with continuous boundary values in $C^{0,\beta}(\T)$ for some $0<\beta<1$. Then
    \begin{enumerate}
        \item $\Cau (f\big|_\T) = f$.
        \item $\Cau (\overline{f\big|_\T})=\overline{\underset{\T}{\emph{Avg}}f}=\overline{f(0)}$.
        \item $\Cau(\Rea f\big|_\T) = \frac12f + \frac12\overline{\underset{\T}{\emph{Avg}}f}=\frac12f+\frac12\overline{f(0)}$.
        \item $\Cau(i\Ima f\big|_\T) = \frac12 f-\frac12\overline{\underset{\T}{\emph{Avg}}f}=\frac12f-\frac12\overline{f(0)}$.
    \end{enumerate}
\end{lemma}
\begin{proof}
    The first equation is just the Cauchy integral formula. To get the second equation, we compute for all $w\in \mathbb D$
    \begin{align}
        \overline{(\Cau \overline{f})(w)} &= -\frac1{2\pi i}\int_\T \frac{f(\tau)}{\overline \tau-\overline w}~d\overline\tau\notag\\
        &=-\frac1{2\pi i}\int_\T\frac{f(\tau)}{\frac{1}{\tau}-\overline w}~d\left(\frac{1}{\tau}\right)\notag\\
        &=\frac1{2\pi i}\int_\T\frac{f(\tau)}{\tau(1-\overline w \tau)}~d\tau\notag\\
        &=\frac1{2\pi i }\int_\T \frac{f(\tau)}{\tau}~d\tau\notag\\
        &=\underset{\T}{\text{Avg}}f.
    \end{align}
    For the third equation, we note that $\Cau(\Rea f)-\frac12 f$ is holomorphic on $\mathbb D$, with the following continuous boundary value on $\T$:
    \begin{equation}
        \frac12\underset{\T}{\text{Avg}}(\Rea f )+ \frac12 \Rea f+\frac{i}2\Hil(\Rea f)-\frac12 f=\frac12\underset{\T}{\text{Avg}}(\Rea f )+\frac{i}{2}[\Hil(\Rea f)-\Ima f].
    \end{equation}
    Note that the above boundary value has constant real part. Thus $\Cau(\Rea f)-\frac12 f$ is constant on $\overline{\mathbb D}$. Hence
    \begin{align}
        \Cau(\Rea f)-\frac12 f &= \underset{\T}{\text{Avg}}\left(\Cau(\Rea f)-\frac12 f \right)\notag\\
        &=[\Cau(\Rea f)](0)-\frac12\underset{\T}{\text{Avg}} f\notag\\
        &=\underset{\T}{\text{Avg}}(\Rea f)-\frac12\underset{\T}{\text{Avg}} f\notag\\
        &=\frac12\overline{\underset{\T}{\text{Avg}}f}.
    \end{align}
    The last equation can be proven similarly.
\end{proof}
\begin{rmk}
A more common form of the Titchmarsh theorem states that for $f$ holomorphic on $\mathbb D$ with continuous boundary values in $C^{0,\be}(\T)$, $0<\be<1$,
\begin{equation}
    \Ima f\big|_\T=\mathfrak H(\Rea f\big|_\T)+\underset{\T}{\emph{Avg}}(\Ima f).
\end{equation}
This obviously holds in view of the proof of Lemma \ref{lem: Titchmarsh}.
\end{rmk}

For equivalence of the dynamical boundary condition before and after application of $\Cau$, we also note the following.
\begin{lemma}\label{lem: vanishing}
    Let $f$ be a holomorphic function on $\mathbb D$ with continuous boundary values in $C^{0,\beta}(\T)$ for some $0<\beta<1$, and let $r\in C^{0,\beta}(\T)$ be real-valued. Suppose $f$ is not identically zero. Then 
    $r=0$ if and only if $\Cau(rf\big|_\T)=0$.
\end{lemma}
\begin{proof}
The forward implication is obvious. Now suppose $\Cau(rf\big|_\T)=0$. The Fourier series of $rf\big|_\T$ converges on $L^2(\T)$:
\begin{equation}
    rf\big|_\T(\tau) = \sum_{n=-\infty}^\infty a_n \tau^n.
\end{equation}
Since $\Cau (\tau^n)=w^n$ if $n\ge 0$ and $\Cau (\tau^n)=0$ if $n<0$, we obtain that 
\begin{equation}
    \Cau(rf\big|_\T)(w)=\sum_{n=0}^\infty a_n w^n.
\end{equation}
Since  $\Cau(rf\big|_\T)=0$, we get $a_n=0$ for all $n\ge 0$, and $rf = \overline{h}$ on $\T$, where $h=\displaystyle\sum_{n=1}^{\infty}\overline{a_{-n}}w^n$ is holomorphic on $\mathbb D$ with continuous boundary value and $h(0) =0$. We take the complex conjugate to get $r\overline f = h$ on $\T$. It follows that $\overline {fh}=r|f|^2=fh$ on $\T$. It follows from Lemma \ref{lem: Titchmarsh} that 
\begin{equation}
    fh=\Cau (fh) = \Cau (\overline{fh})=\overline{(fh)(0)}=0.
\end{equation}
Since $f$ is holomorphic and not identically zero, $f$ cannot be zero on any open subset of $\mathbb D$. Hence $h=0$ on a dense subset of $\mathbb D$ and thus on $\overline{\mathbb D}$. It follows that $rf=0$ on $\T$. By the Schwarz reflection principle, $f$ cannot be zero on any open subset of $\T$. Thus $r=0$ on $\T$.
\end{proof}
The Cauchy integral commutes with angular rotations. In particular, we have
\begin{lemma}\label{lem: d_a Cau}
    Let $f\in C^{1,\beta}(\T)$ for some $0<\beta<1$, then
    \begin{equation}
        \Cau (f_\alpha)=(\Cau f)_\alpha.
    \end{equation}
\end{lemma}
\begin{proof}
    We only need to check that the two sides are equal at any $w\in \mathbb D$, in which case it is obvious as the Cauchy integral commutes with angular rotations.
\end{proof}
\begin{rmk}
    By \eqref{Plemelj}, we also have $\Hil(f_\alpha)=(\Hil f)_\alpha$. In the following, we sometimes write $\Cau f_\alpha$ and $\Hil f_\alpha$ without distinguishing the order of operations. 
\end{rmk}

We now multiply \eqref{DBC 2} by $z_\alpha$, apply $\Cau$ and use Lemma \ref{lem: Titchmarsh} to get
\begin{equation}\label{DBC 3}
     \Cau\left(z_\alpha \Rea \Phi_t-z_t\Rea \Phi_\alpha-i\sigma\left(\frac{z_\alpha}{|z_\alpha|}\right)_\alpha\right)=0.
\end{equation}
In particular, we have computed 
\begin{align}
    \underset{\T}{\text{Avg}}\frac{(\Phi_\alpha)^2}{z_\alpha}&=\underset{\T}{\text{Avg}}\frac{(iw\Phi')^2}{iwz'}\notag\\
    &=\frac1{2\pi i}\int_\T \frac{i\tau^2(\Phi')^2}{\tau z'}~\frac{d\tau}\tau\notag\\
    &=\frac1{2\pi}\int_\T\frac{(\Phi')^2}{z'}~d\tau\notag\\
    &=0.
\end{align}
For the last step, we used the fact that $z$ is a Riemann mapping and $z'$ has no zeros on $\overline {\mathbb D}$. By Lemma \ref{lem: vanishing}, \eqref{DBC 3} is equivalent to \eqref{DBC 2}. So the problem has been reduced to finding two holomorphic functions $z(w,t)$ and $\Phi(w,t)$ on $\mathbb D$ with sufficiently regular boundary values satisfying \eqref{KBC on T} and \eqref{DBC 3}.

We now look for steady traveling wave solutions. Those are solutions for which $z$ is of the form 
\begin{equation}\label{z ansatz}
    z(w,t)=Z(w)e^{ict}
\end{equation} for some fixed Riemann mapping $Z$ and $c\in \R$. Under this Ansatz for $z$, \eqref{KBC on T} reads
\begin{equation}
    \Ima (icZ\overline{Z_\alpha})=\Ima(\overline{\Phi_\alpha})\quad \text{ on }\T,
\end{equation}
or
\begin{equation}
    i\Ima (\Phi_\alpha)=-ic\Rea (Z\overline{Z_\alpha})=-\frac {ic}2(|Z|^2)_\alpha\quad \text{ on }\T.
\end{equation}
By Lemma \ref{lem: Titchmarsh} and Lemma \ref{lem: d_a Cau} we get 
\begin{equation}\label{Phi_alpha 1}
    \frac12\Phi_\alpha = -\frac{ic}2\Cau(|Z|^2 )_\alpha.
\end{equation}
It follows that
\begin{equation}\label{Phi form 1}
    \Phi(w,t) = -ic \Cau(|Z|^2)(w)+b(t).
\end{equation}
Recall that we have the freedom to add an arbitrary function of the form $c_1+ic_2(t)$ to $\Phi$. Thus we may assume without loss of generality that $b(t)$ is real-valued and $b(0)=0$.
We also have from \eqref{Phi_alpha 1} and \eqref{Plemelj} that
\begin{equation}\label{Phi form 2}
    \Phi_\alpha\big|_\T = \frac c2 \mathfrak H(|Z|^2)_\alpha-\frac{ic}2(|Z|^2)_\alpha\big|_\T.
\end{equation}
Using \eqref{z ansatz}, \eqref{Phi form 1}, \eqref{Phi form 2} with Lemma \ref{lem: Titchmarsh}, we can write \eqref{DBC 3} as 
\begin{equation}\label{DBC 4}
    b'(t)Z_\alpha - \frac{ic^2}2\Cau(Z\Hil(|Z|^2)_\alpha)-i\sigma\Cau\left(\frac{Z_\alpha}{|Z_\alpha|}\right)_\alpha=0.
\end{equation}
Any solution of \eqref{DBC 4} forces $b'(t)$ to be a real constant, which we denote by $-b$. We thus have 
\begin{equation}\label{Phi solution}
    \Phi(w,t)=-ic \Cau(|Z|^2)(w)-bt,
\end{equation}
and
\begin{equation}\label{DBC holo}
    -2ibZ_\alpha+c^2\Cau(Z\Hil(|Z|^2)_\alpha)+2\sigma\Cau\left(\frac{Z_\alpha}{|Z_\alpha|}\right)_\alpha=0.
\end{equation}
We compare \eqref{DBC holo} with (49) in \cite{Dya1}. They are the same equation once we note that $\alpha$ corresponds to $-u$ in their paper. 

As is shown above, \eqref{DBC holo} is a consequence of the real equation \eqref{DBC 2}. 
Later it will be useful to go back to this real equation on $\T$. We record this important step as follows.
\begin{lemma}\label{lemma:ComplexToReal}
    Let $Z$ be a Riemann mapping satisfying $Z(0)=0$ and extending to a $C^{2,\beta}$ diffeomorphism on $\overline{\mathbb D}$. If $Z$ solves \eqref{DBC holo}, then 
    \begin{equation}\label{aug main eq}
        -2ibZ_\alpha+c^2 Z\Hil(|Z|^2)_\alpha+2\sigma\left(\frac{Z_\alpha}{|Z_\alpha|}\right)_\alpha-\frac{ic^2}{4}\frac{[\overline{\Cau(|Z|^2)_\alpha}]^2}{\overline{Z_\alpha}}=0 
    \end{equation}
    on $\T$.
\end{lemma}
\begin{proof}
    Let $Z$ be a Riemann mapping as above solving \eqref{DBC holo}. Define $\Phi$ by \eqref{Phi solution}, and $z$ by \eqref{z ansatz}. It is easy to check by repeating the calculations above that \eqref{KBC on T} holds, and the left hand side of \eqref{DBC 2} equals the left hand side of \eqref{DBC 1}, which is real. Also, $\Cau$ applied to the terms in the brackets of the left hand side of \eqref{DBC 2} gives the left hand side of \eqref{DBC holo}. Since $Z$ is a Riemann mapping, $z_\alpha$ is not identically zero. By Lemma \ref{lem: vanishing}, the terms in the bracket in \eqref{DBC 2} must vanish on $\T$. \eqref{aug main eq} follows.
\end{proof}
% Include discussion of the symmetry of Riemann mapping multiplication by e^(i theta).
We briefly comment on the effect of rotational symmetry of Riemann mapping on solutions to \eqref{DBC holo}. As is alluded to earlier, the Riemann mapping that takes $0$ to the center of mass of the fluid domain is only unique up to a rotation. As a result, if $z(w,t)$ and $\Phi(w,t)$ solve \eqref{KBC on T} and \eqref{DBC 4}, so do $z(we^{i\alpha(t)},t)$, $\Phi(w e^{i\alpha(t)},t)$ for any smooth real-valued function $\alpha(t)$. However, in general such a solution will no longer have the form \eqref{z ansatz}, and thus will not give rise to a solution to \eqref{DBC holo}. If it does, there should exist a Riemann mapping $Z_1$ and a real constant $c_1$ such that 
\begin{equation}
    z(we^{i\alpha(t)},t) = Z_1(w) e^{ic_1 t},
\end{equation}
or
\begin{equation}\label{rot sym 1}
    Z(we^{i\alpha(t)})=Z_1(w) e^{i(c_1-c) t}
\end{equation}
for all $w\in \overline{\mathbb D}$ and all $t$. The rotation function $\alpha(t)$ is either constant or non-constant. If $\alpha(t)=\alpha_0$ is a constant, one obtains $Z_1(w)=Z(we^{i\alpha_0})$ and $c_1=c$. This corresponds to the constant rotation symmetry $(Z(w),c)\to (Z(we^{i\alpha_0}),c)$ of solutions to \eqref{DBC holo}. Next assume $\alpha(t)$ is non-constant. By the constant rotation symmetry shown above, we can assume $\alpha(0)=0$ without loss of generality, and then $\alpha(t_0)\ne 0$ for some $t_0$. By continuity of $\alpha(t)$, there exists $t_n$ such that $\alpha(t_n)=\frac{\alpha(t_0)}{n}$ for all positive integers $n$. Evaluating \eqref{rot sym 1} at $t=0$ we get $Z(w)=Z_1(w)$. Evaluating it at $t=t_n$ we get
\begin{equation}
    \left|Z\left(we^{i\frac{\alpha(t_0)}{n}}\right)\right|=|Z(w)|.
\end{equation}
This implies that 
\begin{equation}
    \left|Z\left(e^{i\frac{m}{n}\alpha(t_0)}\right)\right|=|Z(1)|
\end{equation}
for all positive rational numbers $\frac mn$. By continuity we see that $Z$ maps $\T$ to another circle centered at the origin. Recall by our earlier convention that $Z(0)=0$. Thus there exists a $\lambda>0$ and $\alpha_0\in\mathbb R$ such that
\begin{equation}\label{dilation Riem map}
    Z(w) = \lambda e^{i\alpha_0}w.
\end{equation}
It follows from \eqref{rot sym 1} again that $\alpha(t)=(c_1-c)t$. This corresponds to the symmetry $(Z(w),c)\to (Z(w),c_1)$ of solutions to \eqref{DBC holo}. Note, however, that this symmetry occurs only when $Z$ is the simple Riemann map given by \eqref{dilation Riem map}.

\section{Formulation for Bifurcation Analysis}

In view of previous numerical works on this problem, we look for solutions to \eqref{DBC holo} in the case $b=\sigma$. This particular choice of $b$ includes the unit disc $Z_0(w)=w$. We point out that this is not the only reasonable choice of $b$. For instance, instead of fixing $b=\sigma$, one could allow $b$ to vary in the solutions so as to keep the area of the fluid domain fixed. We leave such other considerations for future studies. Note that we can further take $\sigma=1$ without loss of generality, and reduce \eqref{DBC holo} to 
\begin{equation}\label{DBC main}
    2\Cau\left(\frac{Z_\alpha}{|Z_\alpha|}\right)_\alpha-2iZ_\alpha+c^2\Cau(Z\Hil(|Z|^2)_\alpha)=0,
\end{equation}
since having a general $\sigma$ is equivalent to replacing $c^2$ by $\frac{c^2}{\sigma}$. We now define $\F(Z,c)$ to be the left hand side of \eqref{DBC main}:
\begin{equation}\label{def: F}
    \F(Z,c)=2\Cau\left(\frac{Z_\alpha}{|Z_\alpha|}\right)_\alpha-2iZ_\alpha+c^2\Cau(Z\Hil(|Z|^2)_\alpha),
\end{equation}
so that \eqref{DBC main} can be rewritten as $\F(Z,c)=0$.
To study mapping properties of $\F$, we need certain spaces of holomorphic functions on $\mathbb D$. Let $H(\overline{\mathbb D})$ be the space of functions that are holomorphic on $\mathbb D$ with continuous boundary values on $\T$. For the remainder of the paper, we fix the symmetry multiplicity $m\in\mathbb N$, $m\ge 2$. For $k\in \mathbb N\cup \{0\}$, and $0<\beta<1$, we define
\begin{equation}
    X^{k,\beta}=\left\{f\in H(\overline{\mathbb{D}}):f\big|_\T\in C^{k,\beta}(\T),f(w)=\sum_{n=0}^\infty a_n w^{mn+1}, ~a_k\in\R\right\}
\end{equation}
with norm $\|f\|_{X^{k,\beta}}=\|f\big|_\T\|_{C^{k,\beta}(\T)}$. Note that $f\in H(\overline{\mathbb D})$ and $f\big|_\T\in C^{k,\beta}(\T)$ imply that $f\in C^{k,\beta}(\overline{\mathbb D})$.
It is easy to see that the symmetry conditions on $X^{k,\beta}$ can be characterized more instrinsically as 
\begin{equation}
\label{eqn:Syms}
    f(e^{i\frac{2\pi}{m}}w)=e^{i\frac{2\pi}{m}}f(w),~\overline{f(w)}=f(\overline{w}).
\end{equation}
As is well known, the $C^{k,\beta}(\T)$ norm can be characterized in terms of Littlewood-Payley projectors. For $f\in C^{k,\beta}(\T)$ with $f(\tau)=\sum_{n=-\infty}^\infty a_n \tau^{n}$ and $j\in \mathbb N$, let $(\Delta_0 f)(\tau)=\sum_{|n|\le 1}a_n\tau^n$ and 
\begin{equation}
    (\Delta_j f)(\tau)=\sum_{2^j\le |n|< 2^{j+1}}a_n \tau^{n}.
\end{equation}
There exists a $C>0$ such that for all $f\in C^{k,\beta}(\T)$,
\begin{equation}
    \tfrac1C\|f\|_{C^{k,\beta}(\T)}\le \sup_{j\ge 0}2^{j(k+\beta)}\|\Delta_j f\|_{L^\infty(\mathbb{T})}\le C\|f\|_{C^{k,\beta}(\T)}.
\end{equation}

\subsection{Mapping properties of $\F$}
For proper definition of $\F$ and later studies of global continuation, we will restrict $\F$ to certain subsets of $X^{k,\beta}$. First, we need a condition that guarantees the functions will be proper Riemann maps used to parametrize the fluid domain. That is supplied by the chord-arc condition:
\begin{equation}\label{chord-arc}
    \inf_{\tau_1\ne\tau_2\in\T}\left|\frac{f(\tau_1)-f(\tau_2)}{\tau_1-\tau_2}\right|>0.
\end{equation}
We have
\begin{lemma}
    Let $f\in H(\overline{\mathbb{D}})$, $f\big|_\T\in C^{1,\beta}(\T)$. Then the following are equivalent:
    \begin{enumerate}[label=(\alph*)]
        \item $f$ satisfies \eqref{chord-arc}.
        \item $f$ is injective on $\T$ and $f'(\tau)\ne 0$ for all $\tau\in \T$.
        \item $f$ is a diffeomorphism from $\overline{\mathbb D}$ onto $f(\overline{\mathbb D})$.
    \end{enumerate}
\end{lemma}
\begin{proof}
    It is obvious that (a) and (b) are equivalent. It is also obviously that (c) implies (b). It remains to show that (b) implies (c). If (b) holds, $f\big|_\T$ is a $C^{1,\beta}$ Jordan curve. By the Jordan curve theorem, $f(\T)$ separates $\C$ into the interior region $\text{int}(f(\T))$ and the exterior region $\text{ext}(f(\T))$. For any $z\in \text{int}(f(\T))$, the number of points $w\in\disc$ satisfying $f(w)=z$ is equal to the winding number of $f\big|_\T$ around $z$, which is one. Similarly, if $z\in \text{ext}(f(\T))$, the number of points $w\in \disc$ satisfying $f(w)=z$ is equal to the winding number of $f\big|_\T$ around $z$, which is zero. If there exists a $w\in \disc$ such that $f(w)\in f(\T)$, by the open mapping theorem, there exists a $w'\in \disc$ close to $w$ such that $f(w')\in \text{ext}(f(\T))$. This contradicts the above result. Thus $f(\disc)=\text{int}(f(\T))$, $f(\T)=\partial (\text{int}(f(\T)))$, $f:\overline\disc\to f(\overline\disc)$ is bijective, and $f'$ can be continuously extended to $\overline \disc$ such that $f'\ne 0$ on $\overline\disc$. It follows that $f^{-1}$ can be continuously extended to $f(\overline\disc)$ such that $(f^{-1})'\ne 0$ on $f(\overline\disc)$. 
\end{proof}
We thus define 
\begin{equation}
    U^{k,\beta}=\{f\in X^{k,\beta}~|~f \text{ satisfies the chord-arc condition }\eqref{chord-arc}\}.
\end{equation}

It shall be useful in the sequel to have a quantitative measure of the chord-arc condition \eqref{chord-arc}. To that end, we define, for $\tau_1,\tau_2\in\T$,
\begin{equation}
    \label{eqn:QuantChord-Arc}
    \CA(f)(\tau_1,\tau_2) = \begin{dcases} \frac{\tau_1-\tau_2}{f(\tau_1) - f(\tau_2)} & \tau_1\neq\tau_2\\ \frac{1}{\p_\al f(\tau)} & \tau_1=\tau_2=\tau \end{dcases}.
\end{equation}
Then, $f\in X^{k,\be}$ will satisfy the chord-arc condition \eqref{chord-arc} if and only if
\[
\norm{\CA(f)}_{L^\infty(\T^2)} < +\infty.
\]

\begin{lemma}\label{lem: F-deriv}
    For $k\ge 2$, $\F$ is analytic from $U^{k,\beta}\times \R$ to $X^{k-2,\beta}$. The Fr\'echet derivative $D_Z\F$ is given by 
    \begin{equation}
        \label{eqn:FrechetDerivativeF}
D_Z \F(Z,c)[\zeta]  = \Cau\left( \frac{\z_\alpha}{\abs{Z_\alpha}} - \frac{Z_\al^2\overline{\z_\al}}{\abs{Z_\al}^3} \right)_\alpha-2i \zeta_\al + c^2\Cau\left( \zeta\Hil(\abs{Z}^2)_\al+Z\Hil( \zeta\overline{Z} + \overline{\z}Z )_\al \right) .
    \end{equation}
\end{lemma}
We will break the proof of this lemma into several steps. First we show that $\F$ maps into $X^{k-2,\beta}$. The $C^{k-2,\beta}(\T)$ regularity is fairly obvious, as H\"older regularity is preserved under multiplication, square root of strictly positive functions, division by strictly positive functions, $\Cau$ and $\Hil$. Analyticity also follows as the above mentioned maps are all analytic with respect to the relevant H\"older norms. We only need to show the symmetry properties of $X^{k,\beta}$ is perserved under $\F$. To that end, we need the following symmetry lemmas. 
\begin{lemma}\label{lem: rot sym}
    If $f(e^{i\alpha_0}\tau)=e^{i\alpha_0}f(\tau)$ on $\T$ for some $\al_0\in \R$, then 
    \begin{equation}\label{rot sym conversion 1} 
        f_\al(e^{i\alpha_0}\tau)=e^{i\alpha_0}f_\al(\tau),\quad (\Cau f)(e^{i\al_0}w)=e^{i\al_0}(\Cau f)(w),\quad (\Hil f)(e^{i\al_0}\tau)=e^{i\al_0}(\Hil f)(\tau).
    \end{equation} Similarly, if $g(e^{i\alpha_0}\tau)=g(\tau)$ on $\T$ for some $\al_0\in \R$, then 
    \begin{equation}\label{rot sym conversion 2} 
        g_\al(e^{i\alpha_0}\tau)=g_\al(\tau),\quad (\Cau g)(e^{i\al_0}w)=(\Cau g)(w),\quad (\Hil g)(e^{i\al_0}\tau)=(\Hil g)(\tau).
    \end{equation} 
\end{lemma}
\begin{proof}
    The first equation in \eqref{rot sym conversion 1} follows from direction calculation. Denote by $R$ the rotation map $w\mapsto e^{i\al_0}w$. The condition on $f$ can be written as $f\circ R=e^{i\al_0}f$. We can assume $\al_0\ne 0 \mod 2\pi$, otherwise there is nothing to prove. This implies $\underset{\T}{\text{Avg}}~f=0$. Since the Cauchy integral $\Cau$ commutes with rotation, we have 
    \begin{equation}
        (\Cau f)\circ R=\Cau(f\circ R)=\Cau(e^{i\al_0}f)=e^{i\al_0}\Cau f,
    \end{equation}
    which is equivalent to the middle equation in \eqref{rot sym conversion 1}. The last equation in \eqref{rot sym conversion 1} now follows from \eqref{Plemelj}. \eqref{rot sym conversion 2} can be proven similarly.
\end{proof}
\begin{lemma}\label{lem: conj sym}
    If $\overline{f(\tau)}=f(\overline\tau)$ on $\T$, then 
    \begin{equation}\label{conj sym conversion 1}
        \overline{f_\alpha(\tau)}=-f_\alpha(\overline\tau),\quad \overline{(\Cau f)(w)}=(\Cau f)(\overline{w}),\quad \overline{(\Hil f)(\tau)}=-(\Hil f)(\overline\tau).
    \end{equation}
    Similarly, if $\overline{g(\tau)}=-g(\overline\tau)$ on $\T$, then 
    \begin{equation}\label{conj sym conversion 2}
        \overline{g_\alpha(\tau)}=g_\alpha(\overline\tau),\quad \overline{(\Cau g)(w)}=-(\Cau g)(\overline{w}),\quad \overline{(\Hil g)(\tau)}=(\Hil g)(\overline\tau).
    \end{equation}
    
\end{lemma}
\begin{proof}
    The first equation in \eqref{conj sym conversion 1} follows from direct calculation. To get the second equation, we compute
    \begin{align*}
        \overline{(\Cau f)(w)} &= \overline{\frac1{2\pi i}\int_\T \frac{f(\tau)}{\tau-w}~d\tau}\\
        &=-\frac1{2\pi i}\int_\T \frac{f(\overline\tau)}{\overline\tau-\overline{w}}~d\overline\tau\\
        &=-\frac1{2\pi i}\int_{-\T}\frac{f(\tau)}{\tau-\overline{w}}~d\tau\\
        &=(\Cau f)(\overline{w}).
    \end{align*}
    Note that the condition on $f$ implies $\overline{\underset{\T}{\text{Avg}}~f}=\underset{\T}{\text{Avg}}~f$. The last equation in \eqref{conj sym conversion 1} now follows from \eqref{Plemelj}. \eqref{conj sym conversion 2} can be proven similarly.
\end{proof}

\begin{proof}[Proof of Lemma \ref{lem: F-deriv}]
    As is explained above, the $C^{k-2,\beta}(\T)$ regularity of $\F(Z,c)\big|_\T$ and the analyticity of $\F$ with respect to the H\"older norms are standard. \eqref{eqn:FrechetDerivativeF} follows from a direct calculation. $\F(Z,c)$ is obviously in $H(\overline{\mathbb D})$ by construction. We only need to show that $\F(Z,c)$ satisfies the symmetry requirements of $X^{k-2,\beta}$. Since $Z\in X^{k,\beta}$, we have $Z(e^{i\frac{2\pi}{m}}w)=e^{i\frac{2\pi}{m}}Z(w)$ and $\overline{Z(w)}=Z(\overline{w})$ on $\overline\disc$. It follows from Lemma \ref{lem: rot sym}  that $Z_\al(e^{i\frac{2\pi}{m}}w)=e^{i\frac{2\pi}{m}}Z_\alpha(w)$, $|Z_\al|(e^{i\frac{2\pi}{m}}\tau)=|Z_\al|(\tau)$, $\left(\frac{Z_\al}{|Z_\al|}\right)_\al(e^{i\frac{2\pi}{m}}\tau)=e^{i\frac{2\pi}{m}}\left(\frac{Z_\al}{|Z_\al|}\right)_\al(\tau)$, $(|Z|^2)_\al(e^{i\frac{2\pi}{m}}\tau)=(|Z|^2)_\al(\tau)$, $\left[Z\Hil(|Z|^2_\al)\right](e^{i\frac{2\pi}{m}}\tau)=e^{i\frac{2\pi}{m}}\left[Z\Hil(|Z|^2_\al)\right](\tau)$. This implies that $\F(Z,c)(e^{i\frac{2\pi}{m}}w)=e^{i\frac{2\pi}{m}}\F(Z,c)(w)$. It follows from Lemma \ref{lem: conj sym} that $\overline{Z_\alpha(w)}=-Z_\al(\overline{w})$, $|Z_\al(w)|=|Z_\al(\overline{w})|$, $\overline{\left(\frac{Z_\al}{|Z_\al|}\right)_\alpha(\tau)}=\left(\frac{Z_\al}{|Z_\al|}\right)_\alpha(\overline\tau)$, $\overline{(|Z|^2)_\al(\tau)}=-(|Z|^2)_\al(\overline\tau)$, $\overline{\left[Z\Hil(|Z|^2_\al)\right](\tau)}=\left[Z\Hil(|Z|^2_\al)\right](\overline\tau)$. This implies $\overline{\F(Z,c)(w)}=\F(Z,c)(\overline w)$. The proof is complete.
\end{proof}

We next prove certain compactness properties of $\F$.

\begin{lemma}\label{lem: Fredholm}
    Let $k\ge 2$. For any $(Z,c)\in U^{k,\beta}\times\R$, $D_Z\F(Z,c):X^{k,\beta}\to X^{k-2,\beta}$ is Fredholm with index zero.
\end{lemma}

Before beginning the proof of Lemma \ref{lem: Fredholm}, we need to introduce a couple of tools which we will utilize. First, we have the following simple lemma:
\begin{lemma}
    \label{lemma:Holo-Antiholo}
    Let $f,g\in H(\overline\disc)$ with $f(0)=0$ or $g(0)=0$, and let $r\in C(\T)$ be real-valued. If 
    $f = r\overline{g}$ on $\T$, then $f\equiv0$ in  $\overline{\disc}$.
    
\end{lemma}
\begin{proof}
    If $f=r\overline{g}$ on $\T$, then $fg$ is continuous on $\overline{\disc}$, holomorphic on $\disc$ and real-valued on $\T$. It then follows that $fg$ is constant on $\overline{\disc}$. Since $f(0)=0$ or $g(0)=0$, we have
    \[
    fg\equiv0 \text{ on } \overline{\disc}.
    \]
    If $f$ is not identically zero on $\overline\disc$, then $g$ is. Since $f=r\overline{g}$, this implies $f\equiv 0$ on $\mathbb T$ and thus on $\overline \disc$. In other words, $f$ must be zero on $\overline\disc$.
\end{proof}

We will also need to work with the following scale of spaces:
\begin{equation}
    \label{eqn:HolderSymSpaces}
    C_\mathrm{sym}^{k,\be}(\T) \coloneqq \set{ f \in C^{k,\be}(\T) : f(\tau) = \sum_{n\in\Zbb} a_n \tau^{mn+1}, \ a_n \in \R }.
\end{equation}
 We shall equip $C_\mathrm{sym}^{k,\be}(\T)$ with the same norm as that of $X^{k,\be}$; that is,
\[
\norm{f}_{C_\mathrm{sym}^{k,\be}(\T)} \coloneqq \norm{f\vert_\T}_{C^{k,\be}(\T)}.
\]
Note that the symmetry conditions on $C_\mathrm{sym}^{k,\be}(\T)$ are similar to those on $X^{k,\be}$, and can once again be characterized by \eqref{eqn:Syms}. However, the frequencies in the Fourier series are allowed to be negative. In other words, $C_\mathrm{sym}^{k,\be}(\T)$ is not a space of holomorphic functions. In fact, we have
\begin{equation}\label{symmetry class ext}
    C_\mathrm{sym}^{k,\be}(\T)\cap H(\overline\disc)=X^{k,\beta}.
\end{equation}

We are now ready to address Lemma \ref{lem: Fredholm}. It will be useful to break the proof into a couple of parts and we begin with the following claim:
\begin{lemma}
\label{lemma:LinearLeading-OrderInvertible}
    Let $(Z,c) \in U^{k,\be}\times\R$ for $k\ge 2$. Then, the operator
    \begin{equation}
        \label{eqn:Leading-OrderLinear}
        \z \mapsto \Cau\left( \frac{\z_\al}{\abs{Z_\al}} \right)_\al,
    \end{equation}
    as a map from $X^{k,\be}$ to $X^{k-2,\be}$, is invertible. 
\end{lemma}
\begin{proof}
    Let $D_\al \coloneqq \frac{1}{i}\p_\al$ and observe that we can write
    \[
    \Cau\left( \frac{\z_\al}{\abs{Z_\al}} \right)_\al = -D_\al\Cau(\abs{Z_\al}^{-1}D_\al\z).
    \]
    Given that $D_\al$ is clearly invertible from $X^{k,\be}$ to $ X^{k-1,\be}$ and from $X^{k-1,\be}$ to $ X^{k-2,\be}$ (see \eqref{d_al}), we need only consider the invertibility of the operator $\mathfrak S:X^{k-1,\beta}\to X^{k-1,\beta}$ given by
    \begin{equation}
        \label{eqn:Leading-OrderLinearReduced}
        \mathfrak S(\z)= \Cau\left(\frac{\z}{\abs{Z_\al}}\right).
    \end{equation}
    Our strategy is to prove invertibility of the extended operator $\mathfrak{T}: C_\mathrm{sym}^{k-1,\be}(\T) \to C_\mathrm{sym}^{k-1,\be}(\T)$ given by
    \begin{equation}
        \label{eqn:Leading-OrderLinearReducedAlt}
        \mathfrak{T}(\z) = \Cau\left(\frac{\z}{\abs{Z_\al}}\right) + \frac{(\id - \Cau)}{\abs{Z_\al}}\z.
    \end{equation}
    We note that it is straightforward to verify that $\mathfrak{T}$ respects the symmetries of $C_\mathrm{sym}^{k-1,\be}(\T)$ (via Lemmas \ref{lem: rot sym} and \ref{lem: conj sym}). As $(\id-\Cau)\zeta=0$ if $\zeta$ is holomorphic, $\mathfrak T\big|_{X^{k-1,\beta}}=\mathfrak S$. Moreover, $\mathfrak{T}$ is Fredholm of index zero, which is most apparent upon writing
    \[
    \mathfrak{T}(\z) = \frac{\z}{\abs{Z_\al}} + \left[\Cau,\abs{Z_\al}^{-1}\right](\z).
    \]
    The commutator is compact on $C_\mathrm{sym}^{k-1,\be}(\T)$ by Lemma 6.3 of \cite{GoKr1}, while the first term is plainly invertible given $Z\in U^{k,\be}$. Invertibility of $\mathfrak T$ will now follow upon showing that the kernel is trivial.

    Let us now suppose that we have $\z\in C_\mathrm{sym}^{k-1,\be}(\T)$ such that $\mathfrak{T}(\z) = 0$. Note that $$\Cau\left(\frac{\z}{\abs{Z_\al}}\right), \overline{(\id-\Cau)\zeta}\in H(\overline\disc), \quad\text{and } [(\id-\Cau)\zeta](0)=0.$$
    Applying Lemma \ref{lemma:Holo-Antiholo}, we will then have
    \begin{equation}\label{eq: kernel 1}
        \Cau\left(\frac{\z}{\abs{Z_\al}}\right) = 0,
    \end{equation}
    and so
    \begin{equation}\label{eq: kernel 2}
        \frac{(\id-\Cau)}{\abs{Z_\al}}\z = 0.
    \end{equation}
    \eqref{eq: kernel 2} implies that $\Cau\z = \z$ and so, from \eqref{eq: kernel 1}, we have
    \[
    0 = \left(\Cau\left(\frac{\z}{\abs{Z_\al}}\right),\z\right)_{L^2(\T)} = \left(\frac{\z}{\abs{Z_\al}},\Cau\z\right)_{L^2(\T)}=\int_0^{2\pi}\frac{|\zeta|^2}{|Z_\alpha|}~d\al.
    \]
    This implies that $\zeta=0$, and we must have
    $
    \ker\mathfrak{T} = \{0\}$.

    We now want to prove invertibility of $\mathfrak S=\mathfrak{T}\big|_{X^{k-1,\be}}$. Fix some $f\in X^{k-1,\be}\subset C_\mathrm{sym}^{k-1,\be}(\T)$. By invertibility of $\mathfrak T$ on $C_\mathrm{sym}^{k-1,\be}(\T)$, there exists a unique $\z \in C_\mathrm{sym}^{k,\be}(\T)$ such that $\mathfrak{T}(\z) = f$. In other words,
    \begin{equation}
        f-\Cau\left(\frac{\z}{\abs{Z_\al}}\right) = \frac{(\id - \Cau)}{\abs{Z_\al}}\z.
    \end{equation}
    We have
    \begin{equation}
        f-\Cau\left(\frac{\z}{\abs{Z_\al}}\right), \overline{(\id-\Cau)\zeta}\in H(\overline\disc), \quad\text{and } [(\id-\Cau)\zeta](0)=0.
    \end{equation}
    Lemma \ref{lemma:Holo-Antiholo} again implies that $f - \Cau\left(\frac{\z}{\abs{Z_\al}}\right) = 0$,
    and so $(\id-\Cau)\zeta=0$.
    Therefore, we must in fact have $\z\in H(\overline\disc)$. \eqref{symmetry class ext} implies $\z\in X^{k-1,\beta}$.
    This completes the proof.

\end{proof}

We also need to state a standard result on compact embeddings of H\"{o}lder spaces. Let $C_*^r(\T)$, $r\in\R$, denote the scale of Zygmund spaces on $\T$. For $k\in\N$ and $\be\in(0,1)$, it is well known that
\begin{equation}
    \label{eqn:HolderZygmund}
    C_*^{k+\be}(\T) = C^{k,\be}(\T).
\end{equation}
We have the following result regarding embeddings of $C_*^r(\T)$:
\begin{lemma}
\label{lemma:ZygmundEmbedding}
Let $r\in\R$ and $\epsilon>0$. Then,
\[
C_*^{r+\epsilon}(\T) \hookrightarrow C_*^r(\T).
\]
Moreover, this embedding is compact.
\end{lemma}
\begin{proof}
Fix $r\in\R$ and $\epsilon>0$. Let $\Lam(D_\al) \coloneqq (\id - \D_\al)^\frac{1}{2}$. Then, for $\rho\in\R$, $\Lam^\rho(D_\al): C_*^r(\T) \to C_*^{r-\rho}(\T)$ is an isomorphism (see \cite{Tay1}). Moreover, we note that we have the continuous embedding $C_*^r(\T) \hookrightarrow C_*^\rho(\T)$ for any $\rho\leq r$ (e.g., \cite{ScTr1}). Via the above results and \eqref{eqn:HolderZygmund}, the claim will follow upon showing that $C^{0,\epsilon}(\T) \hookrightarrow C_*^0(\T)$ for $0<\epsilon<1$. The desired embedding factors through $C^0(\T)$:
\[
C_*^\epsilon(\T) = C^{0,\epsilon}(\T) \hookrightarrow C^0(\T) \hookrightarrow C_*^0(\T),
\]
where the first (continuous) embedding is compact by Arzel\`{a}-Ascoli.
\end{proof}

We are now ready for the proof of Lemma \ref{lem: Fredholm}:
\begin{proof}[Proof of Lemma \ref{lem: Fredholm}]
 Fix $(Z,c) \in U^{k,\be}\times\R$. We have already shown $D_Z\F(Z,c): X^{k,\be} \to X^{k-2,\be}$. We will proceed by showing that $D_Z\F(Z,c)$ is of the form
 \begin{equation}
     \label{eqn:LinearizationHomeoPlusCpt}
     D_Z\F(Z,c) = H + K,
 \end{equation}
 where $H: X^{k,\be} \to X^{k-2,\be}$ is a homeomorphism and $K: X^{k,\be} \to X^{k-2,\be}$ is compact.
 
 Examining \eqref{eqn:FrechetDerivativeF}, the first term is the leading-order term (second order) and all remaining terms map $X^{k,\be} \to X^{k-1,\be}$. Given that $X^{k-1,\be}$ compactly embeds in $X^{k-2,\be}$ by Lemma \ref{lemma:ZygmundEmbedding}, these lower-order terms all represent compact operators.

 Let us then consider the first term. By Lemma \ref{lemma:LinearLeading-OrderInvertible},
 \[
 \z\mapsto \Cau\left( \frac{\z_\al}{\abs{Z_\al}} \right)_\al
 \]
 is a homeomorphism $X^{k,\be} \to X^{k-2,\be}$. Hence, noting the invertibility of $D_\al: X^{k,\be} \to X^{k-1,\be}$ and $ X^{k-1,\be} \to X^{k-2,\be}$, we are left to prove that the operator
 \begin{equation}
     \label{eqn:Leading-OrderLinearCpt}
     \z \mapsto -\Cau\left( \frac{Z_\al^2\zebar}{\abs{Z_\al}^3} \right).
 \end{equation}
 is compact on $X^{k-1,\beta}$.
 In face, we can easily use Lemmas \ref{lem: rot sym} and \ref{lem: conj sym} to verify that this operator respects the symmetries of $X^{k-1,\be}$ and thus that it maps $X^{k-1,\be}$ to itself. Then, observing that $\Cau\zebar \equiv 0$ for $\z\in X^{k,\be}$, we see that \eqref{eqn:Leading-OrderLinearCpt} has the structure of a commutator:
 \begin{equation}
     \label{eqn:Leading-OrderLinearCpt2}
     \z \mapsto -\left[ \Cau, \abs{Z_\al}^{-3}Z_\al^2 \right](\zebar).
 \end{equation}
 Lemma 6.3 of \cite{GoKr1}, in conjunction with Lemma \ref{lem: d_a Cau}, imply that this commutator, and thus the operator \eqref{eqn:Leading-OrderLinearCpt}, is compact on $X^{k-1,\be}$.

We have now verified \eqref{eqn:LinearizationHomeoPlusCpt} and the proof is thus completed.
\end{proof}

\section{Local Bifurcation}
We are now prepared to construct bifurcation curves of solutions to our traveling wave equation \eqref{DBC main}. These local curves represent small-amplitude traveling waves on our capillary droplet. We shall use a version of the Crandall-Rabinowitz theorem, particularly a version specialized to real analytic operators. We are largely motivated to use this framework, which has been heavily utilized in the study of the water waves problem, due to its powerful global continuation theorem \cite{BuTo1, Dan1, Dan2}. The Crandall-Rabinowitz theorem we utilize is the following:

\begin{thm}
\label{thm:AnalyticCrandall-Rabinowitz}
Let $X, Y$ be Banach spaces and $\F: X \times \R \to Y$ a real analytic mapping with $\F(\xi_0,\lam) = 0 \in Y$ for  all $\lam\in\R$. Moreover, suppose that for some $\lam_0\in\R$, there exists a nonzero $\xi_0\in X$ such that 
\begin{enumerate}[label=(\roman*)]
\item $D_\xi \F(0,\lam_0):X\to Y$ is a Fredholm operator with index zero,
\item $\ker D_\xi \F(0,\lam_0) = \mathrm{span}\{\xi_0\}$,
\item the transversality condition holds:
\[
D^2_{\lam\xi} \F(0,\lam_0)[\xi_0,1]\not\in\text{\emph{ran }} D_\xi \F(0,\lam_0) ,
\]
\end{enumerate}
It then follows that $(0,\lam_0)$ is a bifurcation point. In particular, there exists $\eps>0$ and an analytic map $(\xitil,\lamtil):(\eps,\eps)\to X\times\R$ 
%\begin{equation}
%\label{eqn:Crandall-RabinowitzSolnBranch}
%\left\{ (\xi,\lam) = ( \xitil(s),\lamtil(s) ) : s \in\R \text{ with } \abs{s}<\eps \right\} \subset X\times\R
%\end{equation}
such that $\xitil(0)=0$, $\lamtil(0)=\lam_0$, $\xitil^\pr(0) = \xi_0$,
\[
\F(\xitil(s),\lamtil(s)) = 0 \quad \text{ for all } \abs{s}<\eps,
\]
 and there exists an open set $U \subset X\times\R$ with $(0,\lam_0) \in U$ such that
\[
\{ (\xi,\lam) \in U\setminus(\{0\}\times\R) : \F(\xi,\lam) = 0 \} = \{ ( \xitil(s), \lamtil(s) ) : 0 < \abs{s} < \eps \}.
\]
\end{thm}

The above theorem originally goes back to the work of Crandall-Rabinowitz \cite{CrRa1}. However, they considered operators $\F$ of class $C^k$ and we are interested in analytic $\F$. For a proof of Theorem \ref{thm:AnalyticCrandall-Rabinowitz} in the analytic framework considered here, the interested reader may consult \cite{BuTo1} (in particular, Theorem 8.3.1). Local bifurcation (i.e., various versions of Theorem \ref{thm:AnalyticCrandall-Rabinowitz}) has long and often been applied to the study of water waves. A common application is the problem we consider here: proving the existence of small-amplitude traveling waves in various scenarios (e.g., with constant vorticity, with H\"{o}lder continuous vorticity, with discontinuous vorticity, with or without surface tension, with critical layers, with wind forcing and so on). Recall that a critical layer is a curve along which $\tta_y \equiv 0$ (as a point where $\grad\tta = 0$ is called a stagnation point). For just a few examples of local bifurcation applied to the water waves problem, the interested reader may consult \cite{BDT1,CoSt1,CSV1,Wah1,WBS1} and the references therein.

Though bifurcation theory is a typical strategy for producing traveling wave solutions, particularly for free-surface flows of ideal fluids, it is not the only option available and, in some cases, it cannot be applied. For example, in \cite{LeTi1}, Leoni-Tice demonstrated the existence of viscous traveling gravity-capillary waves (i.e., traveling wave solutions to the free boundary incompressible Navier-Stokes equations), however local bifurcation cannot be applied in this scenario and so the authors are forced to utilize a number of novel analytical techniques in order to overcome the various difficulties posed by the presence of viscosity.

For definiteness, in the remaining of the paper, we will fix the regularity of the domain of $\F$ at $X^{2,\beta}$ in addition to fixing the value of $m$. One could use spaces with higher regularity and obtain analytic dependence of the higher norms, but this is not the main concern of the authors. One can easily verify that $Z_0(w)=w\in U^{2,\beta}$ and $\F(Z_0,c)=0$ for all $c\in \R$. This provides a trivial curve of solutions. In this section, we will be studying our traveling wave equation by constructing a sequence of local curves of solutions, which bifurcate from the trivial solution curve, by applying the theory of local bifurcation. Our first step in this process will be to compute the linearization of $\F$ at the trivial solution curve. 

%\subsection{Linearization at the trivial solution set}
\begin{lemma}
    For $\zeta\in X^{2,\beta}$, denote by $\widehat\z_k$ the (real) Fourier coefficient of $\zeta$ with frequency $mk+1$, i.e.
\begin{equation}
    \zeta(w)=\sum_{k=0}^\infty \widehat\z_k w^{mk+1}.
\end{equation}
Then
\begin{align}
    \left(D_Z\F(Z_0,c)[\zeta]\right)(w)&=2\widehat\z_0 w -\sum_{k=1}^\infty \widehat\z_k(m^2k^2-c^2mk-1)w^{mk+1}.\label{DF(Z0)}\\
    \left(D^2_{cZ}\F(Z_0,c)[\zeta,1]\right)(w)&=2c\sum_{k=1}^\infty\widehat\z_k mkw^{mk+1}.\label{D^2F(Z0)}
\end{align}
\end{lemma}
\begin{proof}
    Using \eqref{eqn:FrechetDerivativeF} on $Z_0(w)=w$, we get 
\begin{equation}
    D_Z\F(Z_0,c)[\zeta]=\Cau\left(\zeta_\alpha+w^2\overline{\z_\al}\right)_\al-2i\zeta_\al+c^2\Cau\left(w\Hil(\zeta \overline w+\overline\z w)_\al\right).
\end{equation} 
Recalling \eqref{d_al} and the simple facts that $\Cau(w^n)=w^n$, $i\Hil(\tau^n)=\tau^n$, $\Cau(w^{-n})=0$, $i\Hil(\tau^{-n})=-\tau^{-n}$ for integer $n\ge 1$, we can easily compute that
\begin{equation}
    \z_\al(w)= iw\z'(w)=i\sum_{k=0}^\infty \widehat\z_k(mk+1)w^{mk+1}.
\end{equation}
\begin{equation}
    \left(\zeta_\alpha+w^2\overline{\z_\al}\right)_\al=-\sum_{k=1}^\infty \widehat\z_k(mk+1)\left[(mk+1)w^{mk+1}-(-mk+1)w^{-mk+1}\right].
\end{equation}
\begin{equation}
    \Hil(\zeta\overline w+\overline\z w)_\alpha=\sum_{k=1}^\infty \widehat\z_k mk(\tau^{mk}+\tau^{-mk}).
\end{equation}
\eqref{DF(Z0)} follows by
noting that $m\ge 2$. \eqref{D^2F(Z0)} follows directly from \eqref{DF(Z0)}. All of the series above converge uniformly on compact subsets of $\disc$ and on $L^2(\T)$.
\end{proof}

We can now state our main local bifurcation theorem. For simple reference, let 
\begin{equation}
\label{eqn:BifurcationValues}
    c_{mk}=\sqrt{mk-\frac{1}{mk}}, \quad k=1,2,3,\dots
\end{equation}
be the bifurcation values. Recall that $m\ge 2$ is a fixed integer. We also denote the solution set and the trivial solution curve by 
\begin{equation}
\label{eqn:SolnSet}
    \mathcal N=\{(Z,c)\in U^{2,\beta}\times \R:\F(Z,c)=0\},\quad \mathcal T=\{(Z_0,c):c\in \R\}.
\end{equation}
\begin{thm}
\label{thm:MainLocal}
    Let $(Z_0,c)\in \mathcal T$ be given. 
    \begin{enumerate}[label=(\roman*)]
        \item If $c\ne \pm c_{mk}$ for any $k$, then no bifurcation occurs near $(Z_0,c)$. In other words, there exists a neighborhood $V$ of $(Z_0,c)$ in $U^{2,\beta}\times\R$ such that $\mathcal N \cap V = \mathcal T\cap V$.
        \item If $c=\pm c_{mk}$ for some $k\ge 1$, then $(Z_0,c)$ is a bifurcation point. In particular, there exists $\eps>0$ and an analytic map $(\widetilde{Z},\widetilde{c}):(-\eps,\eps)\to U^{2,\beta}\times\R$ such that 
        \begin{enumerate}[label=(\alph*)]
            \item $(\widetilde{Z}(s),\widetilde{c}(s))\in \mathcal N$ for all $s\in (-\eps,\eps)$;
            \item $\widetilde{Z}(0)=Z_0$, $\widetilde{c}(0)=c=\pm c_{mk}$, $[\widetilde{Z}'(s)](w)=w^{mk+1}$, $\widetilde{c}'(s)=0$;
            \item there exists an open neighborhood $V$ of $(Z_0,c)$ in $U^{2,\beta}\times\R$ such that $(\mathcal N\setminus\mathcal T)\cap V= \{(\widetilde{Z}(s),\widetilde{c}(s)):0<|s|<\eps\}$.
        \end{enumerate}
        
    \end{enumerate}
\end{thm}
\begin{proof}
    By Lemma \ref{lem: Fredholm}, $D_Z\F(Z_0,c)$ is Fredholm of index zero. If $c\ne\pm c_{mk}$ for any $k$, $m^2k^2-c^2mk-1\ne 0$ for any $k\ge 1$. It follows from \eqref{DF(Z0)} that $D_Z\F(Z_0,c):X^{2,\beta}\to X^{0,\beta}$ is injective. Therefore it is also surjective and is an isomorphism from $X^{2,\beta}$ to $X^{0,\beta}$. By the implicit function theorem, there exists an $\eps>0$ and a neighborhood $V$ of $(Z_0,c)$ in $U^{2,\beta}\times \R$ such that for all $s\in (-\eps,\eps)$, there exists a unique $\widetilde Z(s)$ such that $(\widetilde Z(s),c+s)\in \mathcal N\cap V$. Since $(Z_0,c+s)\in \mathcal T\subset \mathcal N$. The uniqueness part of the above assertion implies that $\widetilde Z(s)=Z_0$. Thus $\mathcal N\cap V=\mathcal T\cap V$.
    
    Now suppose $c=\pm c_{mk}$ for some $k\ge 1$. It is easy to see that condition (i) in Theorem \ref{thm:AnalyticCrandall-Rabinowitz} can be adapted in our case to $D_Z\F(Z_0,c)$ being Fredholm with index zero, which has been verified above. We have $m^2n^2-c^2mn-1=0$ if $n=k$, and $m^2n^2-c^2mn-1\ne 0$ for all other $n\ge 1$. It follows that $\ker D_Z\F(Z_0,c)=\text{span} \{w^{mk+1}\}$, and ran$~D_Z\F(Z_0,c)\subset \ker \langle w^{mk+1},\cdot\rangle\subset X^{2,\beta}$, where $\langle f, g\rangle = \underset{\T}{\text{Avg}}(\overline f g)$ is the $L^2$ inner product on $\T$. Since $\text{ran}~D_Z\F(Z_0,c)$ has codimension 1, we have ran$~D_Z\F(Z_0,c)= \ker \langle w^{mk+1},\cdot\rangle$. By \eqref{D^2F(Z0)}, we have
    $$(D^2_{cZ}\F(Z_0,c)[w^{mk+1},1]=2cmkw^{mk+1}.$$ Since $2cmk\langle w^{mk+1},w^{mk+1}\rangle\ne 0$, we have
    \begin{equation}
        (D^2_{cZ}\F(Z_0,c)[w^{mk+1},1]\ne \text{ran }D_Z\F(Z_0,c).
    \end{equation}
    All the conclusions in part (ii) above now follow from Theorem \ref{thm:AnalyticCrandall-Rabinowitz}, except for the equation $\widetilde{c}'(s)=0$. To see this, we assume for the moment that $k=1$ and take the second $s$-derivative of $\F(\widetilde{Z}(s),\widetilde{c}(s))=0$ and evaluate at $s=0$ to get
    \begin{equation}
        D^2_{ZZ}\F(Z_0,c)[\widetilde{Z}'(0),\widetilde{Z}'(0)]+2\widetilde{c}'(0)D^2_{cZ}\F(Z_0,c)[\widetilde{Z}'(0),1]+D_Z\F(Z_0,c)\widetilde{Z}''(0)=0.
    \end{equation}
    Integrating againt $w^{m+1}$ and recalling that ran $D_Z\F(Z_0,c)=\ker\langle w^{m+1},\cdot\rangle$, we get 
    \begin{equation}\label{c'(0)}
        \widetilde{c}'(0)=-\frac{\langle w^{m+1}, D^2_{ZZ}\F(Z_0,c)[\widetilde{Z}'(0),\widetilde{Z}'(0)]\rangle }{2\langle w^{m+1}, D^2_{cZ}\F(Z_0,c)[\widetilde{Z}'(0), 1]\rangle }
    \end{equation}
    Consider the rotation map $\rot$ defined by $(\rot Z)(w)=e^{-i\frac{\pi}{m}}Z(e^{i\frac{\pi}{m}}w)$. It is easy to see from the symmetry of $X^{k,\beta}$ that $\rot: X^{2,\beta}\to X^{2,\beta}$, and $\rot: X^{0,\beta}\to X^{0,\beta}$. Note that the angle $\frac{\pi}{m}$ above need be chosen specially to make this happen. Furthermore, as $\Cau$ and $\Hil$ commute with rotation and are linear, it follows easily that $\F(\rot Z, c)=\rot\F(Z,c)$. Thus $\F(\rot \widetilde{Z}(s),\widetilde{c}(s))=0$ for all $s\in(-\eps,\eps)$. Taking its second $s$-derivative and evaluating at $s=0$ as before, we obtain an analog of \eqref{c'(0)}, except that all appearances of $\widetilde{Z}'(0)$ are replaced by $\rot \widetilde{Z}'(0)$. Recall that $[\widetilde{Z}'(0)](w)=w^{m+1}$, so we have $$\rot \widetilde{Z}'(0) = e^{-i\frac{\pi}{m}}(e^{i\frac{\pi}{m}}w)^{m+1}=-w^{m+1}=-\widetilde{Z}'(0).$$
    It follows that $\widetilde{c}'(0)=-\widetilde{c}'(0)$, so $\widetilde{c}'(0)=0$. Finally, to deal with the case when $k>1$, we observe that $c_{mk}$ actually only depends on the product $mk$. By uniqueness of the local bifurcation curve established above, the local solutions for the $k>1$ cases are identical to those obtained in the $k=1$ case where $m$ is replaced by the larger integer $mk$, which we have already proven in the above argument.
\end{proof}

\begin{rmk}
    By the above proof, we see that the solutions in the bifurcation curve emanating from $(Z_0,\pm c_{mk})$ actually have $mk$-fold symmetry rather than just $m$-fold symmetry.
\end{rmk}

\section{Global Bifurcation}

As of now, we have, for each $k\in\N$, two local real-analytic curves $\Curveloc^{k,\pm}$ (one corresponding to each choice of sign of the bifurcation values $c=\pm c_{mk}$, $k\in\N$) of small-amplitude solutions to \eqref{DBC holo}. We would now like to continue each of these to a global curve of solutions $\Curve^{k,\pm}$. Our approach to this global continuation problem will be to apply the following (analytic) global bifurcation theorem \cite{CSV1}:

\begin{thm}
\label{thm:AnalyticGlobalBifurcation}
Let $X$ and $Y$ be Banach spaces, with $\open\subset X\times\R$ open and $F:\open\to Y$ a real-analytic map. Suppose that the following hypotheses are verified:
\begin{enumerate}[start=1,label=(H\arabic*)]
\item $(0,\lam)\in\open$ and $F(0,\lam)=0$ for all $\lam\in\R$;
\item for some $\lam_0\in\R$, there exists $0\neq\xi_0\in X$ such that
\begin{enumerate}[label=(\roman*)]
\item $\p_\xi F(0,\lam_0)$ is Fredholm of index zero,
\item $\ker\p_\xi F(0,\lam_0)=\mathrm{span}\{\xi_0\}$,
\item the transversality condition holds:
\[
(\p_\lam\p_\xi F(0,\lam_0))(1,\xi_0)\not\in\Range(\p_\xi F(0,\lam_0));
\]
\end{enumerate}
\item $\p_\xi F(\xi,\lam)$ is Fredholm of index zero for any $(\xi,\lam)\in\open$ such that $F(\xi,\lam)=0$;
\item for some sequence of closed, bounded $Q_j\subset\open$ with $\open = \bigcup_j Q_j$, the set $\{ (\xi,\lam)\in\open : F(\xi,\lam)=0 \} \cap Q_N$ is compact for each $N\in\N$.
\end{enumerate}
Then, there exists a continuous curve $\mathcal{K} = \{ (\xi(s),\lam(s)) : s\in\R \} \subset \open$ of solutions to $F(\xi,\lam)=0$ such that:
\begin{enumerate}[start=1,label=(C\arabic*)]
\item $(\xi(0),\lam(0)) = (0,\lam_0)$;
\item $\xi(s) = \xi_0s + \bigo(s^2)$ in $X$, $\abs{s}<\eps$, as $s\to0$;
\item there exists a neighborhood $\nbhd$ of $(0,\lam_0)$ and $\eps>0$ sufficiently small such that
\[
\{ (\xi,\lam) \in \nbhd : \xi\neq0 \text{ and } F(\xi,\lam) = 0 \} = \{ (\xi(s),\lam(s)) : 0 < \abs{s} < \eps \};
\]
\item $\Curve$ is locally (real) analytic;
\item one of the following alternatives occurs:
\begin{enumerate}[label=(\alph*)]
\item for every $j\in\N$, there exists $s_j>0$ such that $(\xi(s),\lam(s))\not\in Q_j$ for $\abs{s} > s_j$;
\item there exists $\tau>0$ such that $(\xi(s+\tau),\lam(s+\tau)) = (\xi(s),\lam(s))$ for all $s\in\R$.
\end{enumerate}
\end{enumerate}
In addition, such a curve of solutions is unique modulo reparameterizations.
\end{thm}

The proof of this theorem can be found in \cite{CSV1}. Originally, this result goes back to the work of Dancer in the early '70's \cite{Dan1, Dan2}. It was subsequently improved by Buffoni-Toland (see Theorem 9.1.1 in \cite{BuTo1}), however the result as stated there is erroneous. A corrected version of this theorem was then stated and proved by Constantin-Strauss-V\u{a}rv\u{a}ruc\u{a} \cite{CSV1}.

\subsection{Local Compactness of the Solution Set}

Our goal is now to use Theorem \ref{thm:AnalyticGlobalBifurcation} to construct the global curve of solutions for our problem. In order to do so, we will need to verify the condition \textit{(H4)} from Theorem \ref{thm:AnalyticGlobalBifurcation}. Note that conditions \textit{(H1)} and \textit{(H2)} have already been checked as part of constructing $\Curveloc^{k,\pm}$, while \textit{(H3)} follows from Lemma \ref{lem: Fredholm}. Before we begin the work of verifying this condition in our setting, we need to do a bit of setup. We define the following collection of subsets of $U^{2,\be}$:
\begin{equation}
\label{eqn:ClosedSets}
Q_N \coloneqq \{ (\z,c) \in U^{2,\be} \times\R: \norm{(\z,c)}_{X^{2,\be}\times\R} \leq N, \ \norm{\CA(\z)}_{L^{\infty}(\T^2)} \leq N \} \text{ for } N \in \N.
\end{equation}
Notice that $Q_N$ is closed and bounded for each $N\in\N$ with $\bigcup_N Q_N = U^{2,\be}$.

We are now ready to verify hypothesis \textit{(H4)}, which we formulate as follows:
\begin{lemma}
\label{lemma:SolnSetLocCpt}
Take $N\in\N$. Then there exists a $C_N>0$ such that for any $(Z,c)\in\mathcal{N}_N\coloneqq \mathcal{N} \cap Q_N$, we have
\begin{equation}
    \label{eqn:LocalCptEst}
    \|(Z,c)\|_{X^{3,\beta}\times \R} \leq C_N.
\end{equation}
As a result, the set $\mathcal{N}_N$ is compact in $X^{2,\beta}\times\R$ for each $N\in\N$. Recall that $\mathcal{N}$ is defined in \eqref{eqn:SolnSet}.
\end{lemma}
\begin{proof}
Let us take an arbitrary $(Z,c) \in \mathcal N_N$. By the definition of $\mathcal N_N$, We have $\F(Z,c)=0$ and $\|(Z,c)\|_{X^{2,\beta}\times \R}\le N$ and $\left\|\frac1{Z_\alpha}\right\|_{L^\infty(\T)}\le N$. It remains to show $\|Z\|_{C^{3,\beta}(\T)}\le C_N$ for some $C_N>0$. In the following, $C_N$ denotes a generic constants which may be enlarged from step to step. It follows from $\F(Z,c)=0$ and Lemma \ref{lemma:ComplexToReal} that
\begin{equation}
\label{eqn:CurvatureEqn}\left( \frac{Z_\al}{\abs{Z_\al}} \right)_\al = iZ_\alpha - \frac{c^2}{2} Z\Hil(|Z|^2)_\alpha + \frac{ic^2}{8}\frac{[\overline{\Cau(|Z|^2)_\alpha}]^2}{\overline{Z}_\alpha} \quad \text{ on } \T.
\end{equation}
The bounds on $Z$, $c$ and $\frac1{Z_\alpha}$ now imply that 
\begin{equation}
    \left\| \left( \frac{Z_\al}{\abs{Z_\al}} \right)_\al\right\|_{C^{1,\beta}(\T)}\le C_N.
\end{equation} 
This implies that the curvature $\kappa=\frac{1}{iZ_\alpha}\left( \frac{Z_\al}{\abs{Z_\al}} \right)_\al$ also satisfies the bound
\begin{equation}\label{eq: kappa bound}
    \left\|\kappa \right\|_{C^{1,\beta}(\T)}\le C_N.
\end{equation}
Now, let $\sigma(\alpha)=\int_0^\alpha |Z_{\alpha'}(e^{i\al'})|~d\alpha'$ be the arclength parameter of the boundary curve, and denote its inverse by $\al(\sigma)$. Let $\y(\sigma)=Z(e^{i\al(\sigma)})$ be the arclength reparametrization of the boundary curve. Denote by $L=\int_0^{2\pi} |Z_{\alpha'}(e^{i\al'})|~d\alpha'$ its total arclength. We then have
\begin{equation}\label{eq: arclength p bound}
    \frac1{C_N}\le L\le C_N,~\|\sigma(\al)\|_{C^{2,\beta}(\T)}\le C_N, ~\|\alpha(\sigma)\|_{C^{2,\beta}([0,L])}\le C_N.
\end{equation}
Let $k(\sigma)=\kappa(\alpha(\sigma))$ be the curvature in arclength parameter. We have from \eqref{eq: kappa bound} and \eqref{eq: arclength p bound} that
\begin{equation}\label{eq: k curvature bound}
    \|k\|_{C^{1,\beta}([0,L])}\le C_N.
\end{equation}
%Note that $s(\al) \in C_\al^{2,\be}(\T)$ and so we also have $\al(s) \in C_s^{2,\be}([0,L])$. Then, since
%\[
%\y(s) = r(\al(s)),
%\]
%we can conclude that $\y \in C_s^{2,\be}([0,L])$ as a composition of $C^{2,\be}$-regular functions. A similar argument implies that $k=k(s)$ is $C^{1,\be}$-regular with respect to $s$, where $k$ denotes the (local) curvature with respect to the arclength parameterization. 
By the Frenet-Serret formulae, we can write
\begin{equation}
\label{eqn:Frenet-Serret}
\y_{\sigma\sigma} = ik\y_\sigma.
\end{equation}
It follows from \eqref{eq: k curvature bound} and regularity of solutions to \eqref{eqn:Frenet-Serret} that
\begin{equation}\label{eq: 3 beta bound on gamma}
    \|\gamma\|_{C^{3,\beta}([0,L])}\le C_N.
\end{equation}
%But, both terms on the right-hand side of \eqref{eqn:Frenet-Serret}, are $C^{1,\be}$-regular with respect to $s$ and so this implies that $\y_{ss}$ is also $C^{1,\be}$-regular, or that $\y$ is actually $C^{3,\be}$-regular.
The $C^{3,\beta}$ bound of the arclength parametrization does not imply $C^{3,\beta}$ bound of an arbitrary parametrization in general. However, it is enough to recover $C^{3,\beta}$ bound of the {\it conformal} parametrization of the domain and its boundary curve. In fact, we can apply the proof (not just the statement) of Theorem 3.6 in \cite{Pom1} (Kellogg-Warschawski theorem) to deduce
\begin{equation}\label{eqn:LocalCptnessEst}
    \|Z\|_{C^{3,\beta}(\T)}\le C_N
\end{equation}
from \eqref{eq: 3 beta bound on gamma},
noting that the constants in the proof can be made uniform provided that \eqref{eq: arclength p bound}, \eqref{eq: 3 beta bound on gamma} hold. The proof is complete.
\end{proof}

\subsection{The Global Curve of Solutions}

We now arrive at our main result:

\begin{thm}
    \label{thm:MainGlobal}
    Let $(Z_0,\pm c_{mk}) \in \mathcal{T}$ be given, where $c_{mk}$ is given by \eqref{eqn:BifurcationValues}. Then, there exists in $X^{2,\be}\times\R$ a real analytic curve (corresponding to each sign of $\pm c_{mk}$)
    \begin{equation}
    \label{eqn:MainCurve}
        \Curve^{k,\pm} = \{ (\widetilde{Z}(s),\widetilde{c}(s) : s \in \R \}
    \end{equation}
    %of solutions to \eqref{DBC holo} (or, equivalently by Lemma \ref{lemma:ComplexToReal}, equation \eqref{aug main eq}) 
    such that the following hold:
    \begin{enumerate}[label=(\roman*)]
    \item $(\Ztil(s),\ctil(s)) \in \mathcal N$ for all $s\in\R$ (recall that $\mathcal N$, $\mathcal T$ are defined in \eqref{eqn:SolnSet});
    \item $(\Ztil(0),\ctil(0)) = (Z_0,\pm c_{mk})$, $\Ztil^\pr(s)(w) = w^{mk+1}$;
    \item there exists an open neighborhood $V$ of $(Z_0,c)$ in $U^{2,\beta}\times\R$ such that $(\mathcal N\setminus\mathcal T)\cap V= \{(\widetilde{Z}(s),\widetilde{c}(s)):0<|s|<\eps\}$, $\eps>0$ sufficiently small;
    \item one of the following alternatives occurs:
    \begin{enumerate}[label=(\alph*)]
    \item
    \[
    \max\left\{ \norm{\Ztil(s)\vert_\T}_{C^1(\T)}, \norm{\CA(\widetilde{Z}(s))}_{L^\infty(\T^2)} \right\} \to +\infty \text{ as } \abs{s} \to +\infty,
    \]
    \item there exists $p>0$ such that
    \[
    (\Ztil(s+p),\ctil(s+p)) = (\Ztil(s),\ctil(s)) \text{ for all } s \in \R.
    \]
    \end{enumerate}
    \end{enumerate}
    %Moreover, $\Curve^{k,\pm}$ has a real analytic reparameterization about each of its points.
\end{thm}
\begin{proof}
    This result largely follows from Theorem \ref{thm:AnalyticGlobalBifurcation}. We verified the corresponding statements of \textit{(H1)} and \textit{(H2)} in proving the local bifurcation result of Theorem \ref{thm:MainLocal}. Moreover, \textit{(H3)} follows from Lemma \ref{lem: Fredholm} and \textit{(H4)} is given by Lemma \ref{lemma:SolnSetLocCpt}. Theorem  \ref{thm:AnalyticGlobalBifurcation} now implies the above conclusions with \textit{(iv)(a)} replaced by
    \begin{equation}\label{norm blow up condition}
    \max\left\{ \normX{\widetilde{Z}(s)}{2}{\be}, \norm{\CA(\widetilde{Z}(s))}_{L^\infty(\T^2)} \right\} \to +\infty \text{ as } \abs{s} \to +\infty.
    \end{equation}
    Thus, what remains to be proven is the improvement of $X^{2,\beta}$ norm blow-up to $C^1(\T)$ norm blow-up. 

    We shall proceed in stages. Assume we are in the case when \eqref{norm blow up condition} occurs. Furthermore, assume that there exists a $C>0$ such that $\norm{\CA(\widetilde{Z}(s))}_{L^\infty(\T^2)}\le C$. In the following, we always denote such a generic constant by $C$ and may change its value from step to step. We first show that $\normX{\widetilde{Z}(s)}{1}{\be}\to\infty$ as $|s|\to\infty$. If not, there would exist a sequence of solutions $Z_n\in \mathcal N$ such that $\normX{Z_n}{1}{\be}\leq C$, $\norm{\CA(Z_n)}_{L^\infty(\T^2)}\le C$ but $\|Z_n\|_{X^{2,\beta}}\to\infty$. Using  \eqref{eqn:CurvatureEqn} similarly to the proof of Lemma \ref{lemma:SolnSetLocCpt}, we deduce that $\kap_n=\frac{1}{i(Z_n)_\alpha}\left(\frac{(Z_n)_\al}{|(Z_n)_\al|}\right)_\alpha$ satisfies
    \begin{equation}
        \|\kappa_n\|_{C^{0,\beta}(\T)}\le C.
    \end{equation}
    As in the proof of Lemma \ref{lemma:SolnSetLocCpt}, denote the arclength parameter of the boundary curve $Z_n\big|_\T$ by $\sigma_n$, and the inverse parametrization by $\alpha(\sigma_n)$. Denote by $L_n$ the total arclength, and let $\gamma_n(\sigma_n)=Z_n(e^{i\alpha(\sigma_n)})$. We have 
    \begin{equation}
        \frac1C\le L_n\le C,~\|\sigma_n(\alpha)\|_{C^{1,\beta}(\T)}\le C,~\|\alpha(\sigma_n)\|_{C^{1,\beta}([0,L_n])}\le C.
    \end{equation}
    Let $k_n(\sigma_n)=\kappa_n(\alpha(\sigma_n))$ be the curvature in arclength parameter. We have
    \begin{equation}
        \|k_n\|_{C^{0,\beta}([0,L_n])}\le C.
    \end{equation}
    We again deduce from the Frenet-Serret formula that 
    \begin{equation}
        \|\gamma_n\|_{C^{2,\beta}([0,L_n])}\le C,
    \end{equation}
    and from the proof of Theorem 3.6 in \cite{Pom1} that
    \begin{equation}
        \|Z_n\|_{C^{2,\beta}(\T)}\le C.
    \end{equation}
    This contradicts the assumption that $\|Z_n\|_{X^{2,\beta}}\to\infty$.

    %Again, let $\y=\y(s)$ denote the arclength parameterization of the boundary of the fluid domain and take $r(\al) = Z(e^{i\al})$ to be the conformal parameterization. Noting that
    %\[
    %s(\al) = \int_{-\pi}^\al \abs{Z_\al(\al^\pr)} \ d\al^\pr
    %\]
    %is $C^{1,\be}(\T)$-regular, we have that $\al(s)$ is also $C^{1,\be}(\T)$-regular with respect to $s$. Therefore, $\y$ is $r$ composed with a $C^{1,\be}$ change of coordinates and so is $C^{1,\be}(\T)$-regular. Similarly, $k=k(s)$, the curvature with respect to arclength, is $C^{0,\be}(\T)$-regular. Applying the same Frenet-Serret formula (i.e., $\y_{ss} = ik\y_s$) implies that $\y$ is, in fact, $C^{2,\be}$-regular with respect to arclength. The Kellogg-Warschawski theorem (namely, Theorem 3.6 of \cite{Pom1}) then gives the desired estimate:
    %\[
    %\normX{Z}{2}{\be} \leq C_N.
    %\]
    %This implies that it is the $C^{1,\be}(\T)$ norm of $\Ztil(s)$ that must blow up as $\abs{s}\to+\infty$.
    
    Assuming the same setup as in the previous stage, we now show that $\left\|\widetilde{Z}(s)\right\|_{C^1(\T)}\to\infty$ as $|s|\to\infty$. If not, there would exist a sequence of solutions $Z_n\in \mathcal N$ such that $\|Z_n\|_{C^1(\T)}\le C$, $\norm{\CA(Z_n)}_{L^\infty(\T^2)}\le C$ but $\|Z_n\|_{C^{1,\beta}(\T)}\to\infty$. This time we cannot deduce a $C^0(\T)$ bound on $\kappa_n$ directly from \eqref{eqn:CurvatureEqn}, as neither $\Cau$ nor $\Hil$ is bounded on $C^{0}(\T)$. Instead, by the $L^2(\T)$ boundedness of $\Cau$ and $\Hil$, obtain 
    \begin{equation}
        \|\kappa_n\|_{L^2(\T)}\le C,
    \end{equation}
    and
    \begin{equation}
        \frac1C\le L_n\le C,~\|\sigma_n(\alpha)\|_{C^{1}(\T)}\le C,~\|\alpha(\sigma_n)\|_{C^{1}([0,L_n])}\le C.
    \end{equation}
    We deduce similary as before that
    \begin{equation}
        \|k_n\|_{L^2([0,L_n])}\le C.
    \end{equation}
    By the Frenet-Serret formula, we get
    \begin{equation}
        \|\gamma_n\|_{H^2([0,L_n])}\le C.
    \end{equation}
    By the Sobolev embedding theorem, we get
    \begin{equation}
        \|\gamma_n\|_{C^{1,\frac12}([0,L_n])}\le C,
    \end{equation}
    which implies 
    \begin{equation}
        \|Z_n\|_{C^{1,\frac12}(\T)}\le C
    \end{equation}
    similarly as above. If $\beta\le \frac12$, this estimate contradicts the condition $\|Z_n\|_{C^{1,\beta}(\T)}\to\infty$, and the the proof is complete. If not, just apply the previous stage again to obtain 
    \begin{equation}
        \|Z_n\|_{C^{2,\frac12}(\T)}\le C
    \end{equation}
    to finish the proof.
\end{proof}

%\begin{rmk}
%    \label{rmk:Analyticity}
%    
%\end{rmk}

\section{Regularity of the Free Boundary}\label{sec: ana}

In Theorem \ref{thm:MainGlobal}, we have constructed rotational traveling wave solutions whose free boundary curves have $C^{2,\be}$ regularity. However, the boundaries of any such rotational traveling waves are, in fact, real analytic. This is similar to a well-known result of Lewy's in the case of gravity waves on the surface of a fluid layer \cite{Lewy1}. Lewy's result was extended to incorporate the effects of surface tension by Matei in \cite{Mat1}. The proof in \cite{Mat1} can be generalized to our setting and produce the same result.
Indeed, we have the following:
\begin{thm}
Consider a simply connected $C^{2,\beta}$ domain $\CD\subset \R^2$, a real-valued function $\varphi\in C^{2,\beta}(\overline {\CD})$, and constants $c, Q \in \mathbb R$, $\sigma>0$, such that
\begin{align}
    \Delta \varphi &=0 \quad \text{ on }\CD,\\
    (c~(-y,x)- \nabla \varphi )\cdot n&=0 \quad \text{ on }\partial \CD,\\
    -c~(-y,x)\cdot \nabla \varphi +\frac12|\nabla\varphi|^2-\sigma\kappa &= Q \quad \text{ on }\partial \CD.
\end{align}
Here $n$ is the outword unit normal vector on $\partial \CD$ and $\kappa$ is the curvature of $\partial \CD$. Then $\partial \CD$ is real analytic.
\end{thm}
\begin{proof}
    Let $\psi$ be a conjugate harmonic function of $\varphi$ on $\CD$. We have $\varphi_x=\psi_y$, $\varphi_y=-\psi_x$ on $\overline \CD$. Thus
    \begin{equation}
        \nabla\left(\frac{c}{2}(x^2+y^2)+\psi\right)\cdot t = 0 \quad \text{ on }\partial \CD.
    \end{equation}
    Here $t=n^\perp$ is the unit tangent vector on $\partial \CD$. It follows that there exists a constant $\mu\in \R$ such that
    \begin{equation}
        \psi+\frac{c}2(x^2+y^2)=\mu \quad \text{ on }\partial \CD.
    \end{equation}
    Let $\Psi=\psi+\frac{c}2(x^2+y^2)$, we have
    \begin{align}
        \Delta \Psi &=2c \quad \text{ on }\CD,\label{Laplace Psi}\\
    \Psi&=\mu \quad \text{ on }\partial \CD,\label{bdry Psi}\\
    \frac12|\nabla\Psi|^2-\frac{c^2}2(x^2+y^2)-\sigma\kappa &= Q \quad \text{ on }\partial \CD.\label{dyn bdry Psi}
    \end{align}
    If $c=0$, we have from \eqref{Laplace Psi} and \eqref{bdry Psi} that $\psi=\Psi=\mu$ on $\CD$. Thus $\varphi$ is constant on $\CD$, and $\kappa=-\frac{Q}{\sigma}$ on $\partial \CD$. Thus $\partial \CD$ is a circle and is obviously analytic. In the following, we assume $c>0$ without loss of generality. The proof for the case $c<0$ is similar. By \eqref{Laplace Psi}, $\Psi$ is subharmonic. By \eqref{bdry Psi} and the strong maximum principle, $\Psi<\mu$ on $\CD$, and $\nabla\Psi\cdot n>0$ on $\partial \CD$. Let $(x_0,y_0)$ be an arbitrary point on $\partial \CD$. By translation and rotation of coordinates, we may assume without loss of generality that $(x_0,y_0)=(0,0)$ and $\partial \CD$ is tangent to the $x$-axis at $(0,0)$, with $n=(0,1)$ there, while \eqref{dyn bdry Psi} is transformed to 
    \begin{equation}
        \frac12|\nabla\Psi|^2-\frac{c^2}2[(x+x_0)^2+(y+y_0)^2]-\sigma\kappa = Q \quad \text{ on }\partial \CD.
    \end{equation}
    In a neighborhood of $(x,y)=(0,0)$, introduce new coordinates $(p,q)=(x,\Psi(x,y)-\mu)$. Denote by $B_r$ a sufficiently small disc of radius $r$ centered at $(0,0)$. Since $\frac{\partial(p,q)}{\partial(x,y)}(0,0)=\begin{pmatrix}
        1& 0 \\ 0 & \nabla \Psi\cdot n(0,0)
    \end{pmatrix}$ is invertible, $(x,y)\mapsto (p,q)$ maps $B_r\cap \CD$ to an open subset $G$ of the lower half plane, with $B_r\cap \partial \CD$ mapped to a line segment $S$ contained in the $p$-axis. Denote the inverse of the above map by $(x,y)=(p,\eta(p,q))$. It suffices to show that $\eta$ is real analytic on $B_{r'}\cap (G\cup S)$ for some $r'>0$, as $\partial \CD$ near $(x,y)=(0,0)$ is the graph of $y=\eta(x,0)$. By implicit differentiation, we obtain
    \begin{equation}
        \Psi_{x}=-\frac{\eta_p}{\eta_q},~\Psi_y=\frac1{\eta_q},
    \end{equation}
    \begin{equation}
        \Psi_{xx}=-\left(\frac{\eta_p}{\eta_q}\right)_p+\frac{\eta_p}{\eta_q}\left(\frac{\eta_p}{\eta_q}\right)_q,~\Psi_{yy}=\frac1{\eta_q}\left(\frac1{\eta_q}\right)_q,
    \end{equation}
    also, when parametrizing $\kappa$ by $p\subset S$, we have
    \begin{equation}
        \kappa(p) = \frac{\eta_{pp}(p,0)}{(1+\eta_p(p,0))^{\frac32}}.
    \end{equation}
    We obtain that $\eta$ satsifies the following equation and boundary conditions on $G$:
    \begin{align}
        -\left(\frac{\eta_p}{\eta_q}\right)_p+\frac{\eta_p}{\eta_q}\left(\frac{\eta_p}{\eta_q}\right)_q+\frac1{\eta_q}\left(\frac1{\eta_q}\right)_q=2c\quad &\text{ on }G, \label{eq: eta}\\
        -\sigma\frac{\eta_{pp}}{(1+\eta_p)^{\frac32}} +\frac12\left(\frac{\eta_p^2}{\eta_q^2}+\frac{1}{\eta_q^2}\right)-\frac{c^2}{2}\left[(p+x_0)^2+(\eta+y_0^2)\right]=Q\quad &\text{ on } S. \label{bdry cond: eta}
    \end{align}
    Note that $\Psi_x(0,0)=0$, $\Psi_y(0,0)>0$. Thus $\eta_p(0,0)=0$, $\eta_q(0,0)>0$. We linearize \eqref{eq: eta} and \eqref{bdry cond: eta} about $\eta$, freeze the coefficients at $(0,0)$ and take the principle part to get the following linear equation and boundary condition on the lower half plane:
    \begin{align}
        H_{pp}+\frac{1}{\eta_q^2(0,0)}H_{qq}=0\quad & \text{ if }q<0, \label{eq: H}\\
        H_{pp}=0\quad &\text{ if }q=0. \label{bdry cond: H}
    \end{align}
    We verify that \eqref{eq: H} and \eqref{bdry cond: H} form an elliptic equation with coercive boundary condition for $H$ (see \cite{Mat1}, \cite{Kind1} for definition). Ellipticity is obvious. For coerciveness, we look for solutions of the form $H(p,q)=e^{i\lambda p}I(q)$ with $\lambda\ne 0$ that is exponentially decaying for $q<0$. We obtain 
    \begin{align}
        I''(q)-\lambda^2 \eta_q^2(0,0)=0\quad & \text{ if }q<0,\\
        I(0)=0\quad &\text{ if }q=0. 
    \end{align}
    We get $I(q)=C(e^{\lambda \eta_q(0,0)q}-e^{-\lambda \eta_q(0,0)q})$ with $C\in \R$. The exponential decay condition forces $C=0$. Thus the boundary condition \eqref{bdry cond: H} is coercive for \eqref{eq: H}. It follows from Theorem 6.7.6' in \cite{Mor1} that $\eta$ is real analytic in $B_{r'}\cap (G\cup S)$ for some small $r'>0$. The proof is complete.
\end{proof}

\end{document}